\documentclass[11pt]{amsart}
\usepackage{mathrsfs}
\usepackage{amsfonts}
\usepackage{verbatim}
\usepackage{epsfig}
\usepackage{amssymb,mathtools}
\usepackage{amsthm}
\usepackage[all,knot,arc]{xy}
\usepackage{color}
\usepackage{hyperref}
\hypersetup{colorlinks=true,citecolor=blue,linkcolor=blue}
\usepackage{amsfonts,amsmath}
\newtheorem{theorem}{Theorem}[section]
\newtheorem{lemma}[theorem]{Lemma}
\newtheorem{proposition}[theorem]{Proposition}
\newtheorem{corollary}[theorem]{Corollary}

 \usepackage{palatino}

\theoremstyle{definition}
\newtheorem{definition}[theorem]{Definition}

\newtheorem{remark}[theorem]{Remark}

\newcommand{\zee}{\mathbb{Z}}
\newcommand{\arr}{\mathbb{R}}
\newcommand{\cue}{\mathbb{Q}}
\newcommand{\cee}{\mathbb{C}}

\newcommand{\enn}{\mathbb{N}}
\newcommand{\hh}{\mathbb{H}}

\newcommand{\wt}{\widetilde}

\newcommand{\ts}{\textstyle}
\newcommand{\ul}{\underline}

\newcommand{\tG}{\widetilde{G}}

\newcommand{\SLt}{\widetilde{SL}}
\newcommand{\bs}{\backslash}
\renewcommand{\S}{{\mathcal S}}

\def\EtoB{\pi^E_B}

\def\SgtoS2{\pi^{\Sigma}_{S}}

\def\EtoSg{\pi_{\Sigma}^E}

\def\piSL2{\pi_{\widetilde{SL(2,\mathbb{R})}}}

\def\BtoS2{{\pi_{S}^B}}

\DeclareMathOperator{\gr}{gr}
\DeclareMathOperator{\ind}{ind}
\DeclareMathOperator{\isom}{Isom}
\DeclareMathOperator{\im}{Im}
\DeclareMathOperator{\lcm}{lcm}

\numberwithin{equation}{section}

\begin{document}

\title{Cylindrical contact homology of 3-dimensional Brieskorn manifolds}
\author{Sebastian Haney}
\author{Thomas E. Mark}

\begin{abstract} Cylindrical contact homology is a comparatively simple incarnation of symplectic field theory whose existence and invariance under suitable hypotheses was recently established by Hutchings and Nelson. We study this invariant for a general Brieskorn 3-manifold $\Sigma(a_1,\ldots, a_n)$, and give a complete description of the cylindrical contact homology for this 3-manifold equipped with its natural contact structure, for any $a_j$ satisfying $\frac{1}{a_1} + \cdots + \frac{1}{a_n} < n-2$. 
\end{abstract}

\maketitle

\section{Introduction}

Cylindrical contact homology is a part of a much larger structure, symplectic field theory, whose framework was introduced by Eliashberg, Givental and Hofer around 2000 \cite{EGH}. At this point there is a vast literature on symplectic field theory, addressing both applications of the theory and the foundational underpinnings of it; there are still substantial difficulties in defining the most general versions of symplectic field theory. This paper is concerned with a much more specialized situation, for which these issues have been resolved due in large part to the work of Hutchings and Nelson \cite{HN2016, HN2019}. 

We consider a closed 3-dimensional manifold $M$ equipped with a contact structure $\xi = \ker \lambda$, so $\xi$ is a tangent hyperplane distribution defined globally by the vanishing of a 1-form $\lambda$ satisfying $\lambda\wedge d\lambda \ne 0$. The choice of $\lambda$ determines a vector field $R_\lambda$, the Reeb vector field, by the properties $\lambda(R_\lambda) = 1$ and $\iota_{R_\lambda} d\lambda = 0$. As reviewed in more detail in Section \ref{prelims}, the cylindrical contact homology of $(M,\xi)$ is the homology of a chain complex whose generators are periodic orbits of $R_\lambda$, and whose differential counts pseudo-holomorphic cylinders in the symplectization $(\arr\times M, d(e^t\lambda))$ that are asymptotic to Reeb orbits. Making this notion precise and proving that the resulting homology groups are well-defined and independent of choices is the subject of \cite{HN2016, HN2019}, and requires certain assumptions on $M$ and $\xi$. The property that suffices for us is the following.

\begin{definition} A contact 3-manifold $(M, \xi = \ker\lambda)$ is {\em hypertight} if no periodic orbit of the Reeb vector field $R_\lambda$ is contractible.
\end{definition}

Note that a ``periodic orbit'' is not assumed to be parametrized with minimal period; if it is, then it is called a simple orbit. One of the main results of \cite{HN2019} is that the cylindrical contact homology of a hypertight 3-manifold is an invariant of the underlying contact structure $\xi$, i.e. does not depend on the choice of nondegenerate hypertight contact form $\lambda$ (or any of the other attendant choices in the construction). Here ``nondegenerate'' is a condition on $\lambda$ that implies in particular that closed Reeb orbits are isolated; it can be arranged by a small perturbation, but see Section \ref{prelims} for further discussion.

Our main result is a calculation of the cylindrical contact homology in the case that $M$ is a Brieskorn manifold $\Sigma(a_1,\ldots, a_n)$ of hyperbolic type, meaning that $\sum \frac{1}{a_j} < n-2$. Here $a_1,\ldots, a_n$ are arbitrary integers greater than or equal to 2 satisfying the given condition; if they are pairwise relatively prime then $\Sigma(a_1,\ldots, a_n)$ is an integral homology sphere. Let us consider the case $n=3$ for the moment, and recall that $\Sigma(p,q,r)$ is defined to be the link of a complex surface singularity:
\[
\Sigma(p,q,r) = \{(x,y,z)\in \cee^3\,|\, x^p + y^q + z^r = 0 \}\cap \{|x|^2 + |y|^2 + |z|^2 = 1\}.
\]
Each tangent space $T_m \Sigma(p,q,r)$ contains a unique complex line; this field of complex tangencies defines a ``standard'' contact structure $\xi = \xi_{p,q,r}$ on $\Sigma(p,q,r)$, also known as the {\it Milnor fillable} contact structure.

Following \cite{HN2016} we take cylindrical contact homology to be defined with rational coefficients.

\begin{theorem}\label{firstthm} For any pairwise relatively prime triple $p,q,r$ as above, the contact structure $\xi_{p,q,r}$ is hypertight. The cylindrical contact homology of $(\Sigma(p,q,r), \xi_{p,q,r})$ is given by
\[
HC_*(\Sigma(p,q,r), \xi_{p,q,r}) = \bigoplus_{n\geq 1} G(p,q,r)[-2dn] \oplus \bigoplus_{n\geq 1} H_*(S^2)[-2dn-2],
\]
where $d = pqr - qr - pr-pq$.
\end{theorem}
Here $G(p,q,r)$ is a graded $\cue$-vector space of dimension $(p-1) + (q-1) + (r-1)$, whose grading takes values in even integers between $-2$ and $-2d$. One can understand the generators of the homology above as follows. Recall that the action of $S^1$ on $\cee^3$ given by $\zeta\cdot(x,y,z) = (\zeta^{qr}x, \zeta^{pr}y, \zeta^{pq}z)$ preserves $\Sigma(p,q,r)$, and induces on it the structure of a Seifert 3-manifold having three exceptional orbits $\gamma_1$, $\gamma_2$ and $\gamma_3$ corresponding to stabilizers of order $p$, $q$, and $r$ respectively. These exceptional orbits are simple Reeb orbits of a natural contact form defining $\xi_{p,q,r}$ (and of the perturbations that we use, see Sections \ref{brieskornsec} and \ref{perturbsec} for a description of the contact form and its perturbation), hence correspond to generators of the cylindrical chain complex. Moreover, the iterates $\gamma_1^p$, $\gamma_2^q$ and $\gamma_3^r$ are each homotopic to a regular orbit of the circle action. The space $G(p,q,r)$ is spanned by the iterates of $\gamma_1$, $\gamma_2$, $\gamma_3$ up through orders $p-1$, $q-1$, $r-1$ respectively, while (intuitively speaking) the singular homology $H_*(S^2)$ appears because the orbit space of the circle action is topologically a 2-sphere. The additional copies of these spaces with shifted gradings correspond to the same collections of orbits, covered with higher multiplicity.

The strategy of the proof is as follows. When $\frac{1}{p} + \frac{1}{q} + \frac{1}{r}< 1$, it is well known that $\Sigma(p,q,r)$ is a geometric 3-manifold whose universal covering is $\widetilde{SL(2,\arr)}$, the universal covering group of $PSL(2;\arr)$, equipped with a left-invariant metric. In fact, as described by Milnor \cite{milnor1975} and Dolga\v cev \cite{dolgacev}, one can identify $\Sigma(p,q,r)$ with the quotient space of $\widetilde{SL(2,\arr)}$ by a discrete subgroup $\Pi(p,q,r)$. As $PSL(2,\arr)$ can be identified with the unit tangent bundle of the hyperbolic plane $\hh$, there is an identification $\widetilde{SL(2,\arr)} \cong \hh\times\arr$. Moreover, a natural left-invariant contact structure on $\widetilde{SL(2,\arr)}$ admits a contact form that is easily described in terms of the natural coordinates on $\hh\times \arr$ as $\lambda = dt + \frac{1}{y}dx$ (see Theorem \ref{invcontact}; here and throughout we consider the upper half-plane model for $\hh$). The Reeb trajectories of this form are simply vertical lines, and invariance of $\lambda$ yields the contact form on $\Sigma(p,q,r)$ mentioned above. Our strategy is to use certain $\Pi(p,q,r)$-invariant functions described by Milnor to construct a nondegenerate perturbation of $\lambda$ that is explicit enough to allow calculations of the Conley-Zehnder indices. 

Many of the results in Milnor's work were generalized by Neumann to links of Brieskorn complete intersection singularities \cite{neumann}, and our techniques carry over easily to this situation. Recall that if we are given an $n$-tuple of integers $(a_1,\ldots, a_n)$, $a_j\geq 2$, and a generic $(n-2)\times n$ matrix of complex numbers $(m_{ij})$, the 3-manifold $\Sigma(a_1,\ldots, a_n)\subset \cee^n$ is defined by
\begin{equation}\label{brieskorndef}
\Sigma(a_1,\ldots, a_n)= \{(z_1,\ldots z_n) \, | \, \ts\sum_{j=1}^n m_{ij}z_j^{a_j} = 0 \, (i = 1,\ldots, n-2)\}\cap S^{2n-1}.
\end{equation}
Neumann shows that if $\sum \frac{1}{a_j} < n-2$ then there is a description of $\Sigma(a_1,\ldots, a_n)$ as a quotient of $\widetilde{SL(2,\arr)}$ by a group $\Pi(a_1,\ldots, a_n)$ just as before. Moreover, $\Sigma(a_1,\ldots, a_n)$ is a Seifert manifold over an orientable surface $S_g$ of genus $g$, where:
\begin{eqnarray}
g &=& \frac{1}{2}\left(2 + (n-2)\frac{\prod a_i}{\lcm(a_i)} - \sum s_j \right),\label{genusS}\\
s_j &=& \frac{\prod_{i \neq j} a_i}{\lcm_{i\neq j} a_i}.
\end{eqnarray}
Our result is easiest to state in the case that the $a_i$ are pairwise relatively prime. Note that in this case $g = 0$ and $\Sigma(a_1,\ldots, a_n)$ is an integral homology sphere. The theorem above generalizes directly to:

\begin{theorem}\label{introgenthm} For any pairwise relatively prime $a_1,\ldots, a_n$ with $\sum\frac{1}{a_j} < n-2$, the natural contact structure $\xi$ on $\Sigma(a_1,\ldots, a_n)$ is hypertight. The cylindrical contact homology of $(\Sigma(a_1,\ldots, a_n), \xi)$ is given by
\[
HC_*(\Sigma(a_1,\ldots, a_n), \xi) = \bigoplus_{n\geq 1} G(a_1,\ldots, a_n)[-2dn] \oplus \bigoplus_{n\geq 1} H_*(S^2)[-2dn-2],
\]
where 
\[
d = (n-2)\prod_j a_j - \sum_j a_1\cdots\widehat{a_j}\cdots a_n.
\]

\end{theorem}

Here again, $G(a_1,\ldots, a_n)$ is a graded $\cue$-vector space that we can think of as generated by the iterates of the exceptional Seifert fibers $\gamma_j$, $j = 1,\ldots, n$, up to order $a_j - 1$. 

The corresponding statement for $\Sigma(a_1,\ldots, a_n)$ for general $a_j$ is given in Theorem \ref{generalthm} below.

The paper is organized as follows. Section \ref{prelims} gives a more detailed description of the definition of cylindrical contact homology, following Hutchings and Nelson \cite{HN2016}. Section \ref{brieskornsec} reviews Brieskorn spheres from the perspective of 3-dimensional geometry, namely as the quotient of the universal covering group of $SL(2,\arr)$ by a discrete subgroup, following Milnor \cite{milnor1975}. This description allows us to write down a contact form describing the standard contact structure on $\Sigma(p,q,r)$ in terms of coordinates on the universal cover, where the Reeb dynamics are particularly easy to describe. In fact, the Reeb vector field for this form is everywhere tangent to the circle orbits on $\Sigma(p,q,r)$. This is geometrically pretty, but of course the form then fails to be nondegenerate. Our coordinate description allows us to perturb the form to achieve nondegeneracy (up to a fixed action, at least), following ideas in \cite{Ne18}. This is done in section \ref{perturbsec}, where the crucial calculation is the Conley-Zehnder indices of the resulting nondegenerate closed orbits. It is this calculation that determines the grading on cylindrical contact homology, and the perturbation and index calculations are the essentially new contributions of our work. The differential on the complex is mostly trivial for grading reasons, but there are nonzero differentials that we compute by reducing to the situation of \cite{Ne18}.

In section \ref{gensec} we describe the generalization from $\Sigma(p,q,r)$ to arbitrary $\Sigma(a_1,\ldots, a_n)$. 

{\bf Acknowledgements.} Our thanks to Jo Nelson for many helpful conversations. Thomas Mark was supported in part by a grant from the Simons Foundation (523795, TM).

\section{Contact Geometry Preliminaries}\label{prelims}

\begin{definition} Let $M$ be a smooth, oriented manifold of dimension $2n-1$.  A (positive, co-oriented) {\em contact structure} on $M$ is a hyperplane distribution $\xi\subset TM$ that is ``maximally non-integrable'' in the sense that $\xi$ is equal to $\ker(\lambda)$ for a $1$-form $\lambda\in \Omega^1(M)$, such that $\lambda\land(d\lambda)^{n-1}$ is a volume form determining the given orientation of $M$.  Any such form $\lambda$ is called a {\em contact form} for $\xi$.  
\end{definition}

Observe by the required condition on $\lambda$, the restriction of $d\lambda$ to the contact planes determines a symplectic structure on $\xi$.

If $\lambda$ is a contact form on $M$ and $g:M\rightarrow(0,\infty)$ is smooth then $g\lambda$ defines the same contact structure, as $\ker\lambda = \ker g\lambda$.  Given two contact structures $\xi_0 = \ker\lambda_0$ and $\xi_1 =\ker\lambda_1$, we say that they are \textit{contactomorphic} if there is a diffeomorphism $\psi:M\rightarrow M$ such that $\psi_*(\xi_0) = \xi_1$, which is equivalent to saying that $\psi^*\lambda_1 = g\lambda_0$ for some smooth positive function $g$ on $M$.

Given any contact form $\lambda$ on $M$, there is a unique vector field $R_{\lambda}$ satisfying the properties
\begin{description}
\item[(i)] $\lambda(R_{\lambda}) \equiv 1$
\item[(ii)] $d\lambda(R_{\lambda},\cdot) \equiv 0$
\end{description}
This vector field $R_{\lambda}$ is called the \textit{Reeb vector field of $\lambda$}.  Condition \textbf{(i)} implies that $R_{\lambda}$ is always transverse to $\ker\lambda$, but in general for two contact forms $\lambda_0$ and $\lambda_1$ with $\ker\lambda_0 = \ker\lambda_1$, the Reeb vector fields $R_{\lambda_0}$ and $R_{\lambda_1}$ will be different.  A \textit{Reeb orbit of period $T$} of $\lambda$ is defined to be a map $\gamma:\mathbb{R}/T\mathbb{Z}\rightarrow M$ such that $\dot{\gamma}(t) = R_{\lambda}(\gamma(t))$. We do not require $T$ to be the minimal period, but we say that $\gamma$ is {\em simple} if the map $\gamma:\mathbb{R}/T\mathbb{Z}\rightarrow M$ is injective.  For a Reeb orbit $\gamma$ of period $T$ and an integer $k\in \enn$, we define the $k$th iterate $\gamma^k$ to be the composition
\[\mathbb{R}/kT\mathbb{Z}\rightarrow\mathbb{R}/T\mathbb{Z}\xrightarrow{\gamma}M\]
where the first map is the natural quotient.

We denote the flow of the Reeb vector field by $\psi_t$, so by definition $\dot{\psi_t} = R_{\lambda}\circ\psi_t$. It is an easy exercise to see that $\psi_t$ is a contactomorphism for each $t$ (in fact, $\psi_t^*\lambda = \lambda$), and in particular the derivative $d\psi_t$ maps the contact hyperplane $\xi_x$ symplectomorphically to $\xi_{\psi_t(x)}$.

 Let $\gamma$ be a Reeb orbit of period $T$. We say $\gamma$ is {\em nondegenerate} if the linearized return map $d\psi_T:\xi_{\gamma(0)}\rightarrow\xi_{\gamma(T) = \gamma(0)}$ does not have $1$ as an eigenvalue.  If $\lambda$ is a contact form whose Reeb orbits are all nondegenerate, then $\lambda$ is said to be a nondegenerate contact form. Given $\xi$, the set of nondegenerate contact forms is dense among contact forms defining $\xi$.

\subsection{Cylindrical contact homology}\label{backgroundsec}

Let us fix a 3-dimensional contact manifold $(M,\xi)$ and a nondegenerate contact form $\lambda$. Under suitable hypotheses one can define the cylindrical contact homology $CH(M,\xi)$ as the homology of a chain complex whose generators are (certain) periodic Reeb orbits and whose differential ``counts'' pseudoholomorphic cylinders in a symplectization of $M$. We outline the construction here, essentially following Hutchings-Nelson \cite{HN2016}. Note that the chain complex we describe depends on the choice of contact form $\lambda$ as well as an almost-complex structure $J$ on the symplectization of $(M,\lambda)$ (see below), so we write the complex as $CC(M,\lambda, J)$. Recent work of Hutchings and Nelson \cite{HN2019} implies that the homology $CH(M,\xi) = H_*(CC(M,\lambda,J))$ depends only on $\xi$. To simplify matters we assume where convenient that no periodic orbits of $R_\lambda$ are contractible, which will be the case in the examples we consider, and this in particular means that $\lambda$ is {\em dynamically convex} in the terminology of \cite{HN2016}. Moreover, we work throughout with rational coefficients that we omit from the notation.
 
First we describe the Reeb orbits that generate $CC(M,\lambda,J)$. Let $\gamma$ be a Reeb orbit of period $T$ and consider the linearized return map $d\psi_T: \xi_{\gamma(0)}\to \xi_{\gamma(0}$. This is an area-preserving linear isomorphism with respect to the symplectic (area) form $d\lambda$, and we can distinguish three cases depending on its eigenvalues. We say $\gamma$ is \textit{elliptic} if $d\psi_T$ has eigenvalues on the unit circle, \textit{positive hyperbolic} if $d\psi_T$ has positive real eigenvalues, and \textit{negative hyperbolic} if $d\psi_T$ has negative real eigenvalues. A Reeb orbit $\gamma$ is called \textit{bad} if $\gamma$ is an even iterate of a negative hyperbolic orbit.  If $\gamma$ is not bad it is said to be \textit{good}.  Define $CC(M,\lambda,J)$ to be the $\mathbb{Q}$-vector space generated by the good Reeb orbits of $\lambda$. 

In cases of interest to us, for example when $\pi_2(M) = 0$, the complex $CC(M,\lambda,J)$ carries a relative grading by the integers, whose definition makes use of the Conley-Zehnder index of a Reeb orbit. For an orbit $\gamma$, fix a trivialization $\tau$ of $\xi|_{\gamma}$, so that under this trivialization we can regard $\lbrace d\psi_t\rbrace_{t\in[0,T]}$ as a path of symplectic matrices. The Conley-Zehnder index of this path can be defined as follows (cf. \cite[Section 2]{HN2016} or \cite[Section 3.2]{HutchingsECHnotes}). If $\gamma$ is hyperbolic then $d\psi_T$ has an eigenvector $v$, and the path $\lbrace d\psi_t(v)\rbrace_{t\in[0,T]}$ rotates  through an angle $\theta = k\pi$ for some integer $k$. In this case we define $CZ_{\tau}(\gamma) = k$, and note that this is even when $\gamma$ is positive hyperbolic and odd when $\gamma$ is negative hyperbolic.  If $\gamma$ is elliptic then we can adjust $\tau$ by a homotopy so that each $\lbrace d\psi_t\rbrace_{t\in[0,T]}$ is a rotation by the angle $2\pi\theta_t$, where $\theta_t$ depends continuously on $t$ and $\theta_0 = 0$. In this case we let $CZ_{\tau}(\gamma) = 2\lfloor\theta_T\rfloor+1$. Note that $CZ_\tau(\gamma)$ is even if and only if $\gamma$ is a positive hyperbolic orbit.

The Conley-Zehnder index depends only on the homotopy class of the trivialization $\tau$ along $\gamma$ (and its value modulo 2 is independent of the trivialization). In the cases we consider below the contact structure $\xi$ admits a global trivialization on all of $M$, and such a trivialization is unique up to homotopy whenever the first Betti number of $M$ vanishes. In particular, under these circumstances it makes sense to compare the numerical values of $CZ_\tau(\gamma)$ for different $\gamma$, and introduce a grading on $CC(M,\lambda, J)$ by declaring a Reeb orbit $\gamma$ to lie in grading level $\gr(\gamma) := CZ_\tau(\gamma) - 1$ for a global trivialization $\tau$ as above. 

The differential on $CC(M,\lambda,J)$ is defined in terms of the symplectization of $(M,\lambda)$, which is the 4-manifold $\arr\times M$ with the symplectic form $d(e^t\lambda)$, where $t$ is the coordinate on $\arr$. An almost-complex structure $J$ on $\arr\times M$ is compatible with $\lambda$ if it is $\arr$-invariant, carries the coordinate vector $\partial_t$ into the Reeb field $R_\lambda$, preserves the contact planes, and satisfies $d\lambda(v, Jv)>0$ for all nonzero vectors $v\in \xi$. Fix such a $J$, and let $\pi_{\mathbb{R}}$ and $\pi_M$ denote the projection maps from $\mathbb{R}\times M$ onto the $\mathbb{R}$ and $M$ factors, respectively.  A {\em $J$-holomorphic cylinder} from $\gamma^+$ to $\gamma^-$ is a map $u:\mathbb{R}\times S^1\rightarrow\mathbb{R}\times M$ such that:
\begin{itemize}
\item[] $\partial_t u+J\partial_{\theta}u = 0$,
\item[] $\lim_{t\rightarrow\pm\infty}\pi_{\mathbb{R}}(u(t,\theta)) = \pm\infty$, and 
\item[]$\lim_{t\rightarrow\pm\infty}\pi_M(u(t,\cdot))$ is a parametrization of $\gamma^{\pm}$. 
\end{itemize}
The cylinder $u$ is said to have positive end $\gamma^+$ and negative end $\gamma^-$.  Two $J$-holomorphic cylinders with the same positive and negative ends are said to be equivalent if they differ by a translation and rotation of $\mathbb{R}\times S^1$, and the set of such equivalence classes is denoted $\mathcal{M}^J(\gamma^+,\gamma^-)$.  There is a $\mathbb{R}$ action on $\mathcal{M}^J(\gamma^+,\gamma^-)$ given by translation in the $\mathbb{R}$ factor of $\mathbb{R}\times M$.

The {\em Fredholm index} of a map $u: \arr\times S^1\to \arr\times M$ with positive and negative ends at $\gamma^+$ and $\gamma^-$ is defined by
\[
\ind(u) = CZ_{\tau}(\gamma^+)-CZ_{\tau}(\gamma^-)+2c_{1}(u^*\xi,\tau)
\]
where $c_{1}(u^*\xi,\tau)$ is the first Chern class relative to the trivialization $\tau$. (the definition normally includes a term accounting for the Euler characteristic of the domain, which always vanishes for us).  The Chern class term vanishes iff $\tau$ extends to a trivialization of $u^*\xi$, which happens, for instance, when $\tau$ is induced by a global trivialization of $\xi$.  If $J$ is generic and $u$ is somewhere injective then $\mathcal{M}^J(\gamma^+,\gamma^-)$ is a manifold of dimension $\ind(u)$ near $u$.  

We write $\mathcal{M}_k^J(\gamma^+,\gamma^-)$ for the set of $J$-holomorphic cylinders $u$ with $\ind(u) = k$. Following \cite{HN2016}, we define the operator $\delta:CC(M,\lambda,J)\rightarrow CC(M,\lambda,J)$ by
\[
\delta\gamma^+ = \sum_{\gamma^-}\sum_{u\in\mathcal{M}^J_1(\gamma^+,\gamma^-)/\mathbb{R}}\frac{\epsilon(u)}{d(u)}\gamma^-
\]
where $d(u)$ denotes the covering multiplicity of $u$, and $\epsilon(u)\in\lbrace\pm 1\rbrace$ is a sign determined by a choice of coherent system of orientations.\footnote{Taking a different coherent choice of orientations will give a chain homotopic chain complex once $\partial$ is defined, and hence this does not affect the definition of cylindrical contact homology.} (For this definition to be sensible, it is necessary that the spaces $\mathcal{M}_1^J(\gamma^+,\gamma^-)/\mathbb{R}$ are compact and cut out transversely, hence consist of finite sets of points.)  Define another operator $\kappa: CC(M,\lambda,J)\rightarrow CC(M,\lambda,J)$ by $\kappa(\gamma) = d(\gamma)\gamma$ where again $d$ indicates the covering multiplicity.  One expects that $\delta\kappa\delta = 0$, which would imply that the composition
\[
\partial := \delta\kappa
\]
is a differential on $CC(M,\lambda,J)$. The main result of \cite{HN2016} is that, for generic choice of $J$, this construction does give rise to a  well-defined chain complex in the case that $\lambda$ is dynamically convex, in particular if $\lambda$ is hypertight. The cylindrical contact homology is then defined by 
\[
 CH(M,\lambda,J) = H_*(CC(M,\lambda, J), \partial).
 \]
 Moreover, the results of \cite{HN2019} imply that the homology groups do not depend on the choice of hypertight form $\lambda$, nor the choice of suitable $J$ (see Corollary 1.10 of \cite{HN2019}).
 
\section{Brieskorn manifolds and $\widetilde{SL}$ geometry}\label{brieskornsec}
In this section we  describe the perspective we take on Brieskorn manifolds in this paper. For concreteness, we consider the case $n = 3$ in \eqref{brieskorndef}, and write $p,q,r$ for $a_1,a_2,a_3$. We also suppose $p,q,r$ are pairwise relatively prime, returning to the general case in section \ref{gensec}. The Brieskorn manifolds $\Sigma(p,q,r)$ that we study (those for which $\frac{1}{p}+\frac{1}{q} + \frac{1}{r} < 1$) can be described as quotients of $\widetilde{SL(2,\mathbb{R})}$, the universal cover of $PSL(2,\mathbb{R})$, by certain discrete subgroups. For brevity we write $\SLt$ for $\widetilde{SL(2,\mathbb{R})}$. Section \ref{Bri1} describes $\SLt$ as a space of `labeled biholomorphic maps' on the upper half-plane in $\cee$, and introduces an invariant contact form on $\SLt$. Section \ref{Bri2} recalls the construction of a discrete subgroup $\Pi(p,q,r)< \SLt$ such that $\SLt/\Pi(p,q,r) \cong \Sigma(p,q,r)$, and relates this description of $\Sigma(p,q,r)$ with its description as the link of a hypersurface singularity in $\cee^3$. Much of this section is a summary of portions of the paper of Milnor \cite{milnor1975}.

\subsection{Geometry of $\widetilde{SL(2,\mathbb{R})}$}\label{Bri1}

Let $(\mathbb{H},\frac{1}{y^2}(dx^2+dy^2))$ be the upper half-plane with its usual hyperbolic metric.  We remark $\mathbb{H}$ has the structure of a K{\"a}hler manifold with the complex structure it inherits as an open subset of $\cee$, together with the symplectic form $\omega = \frac{1}{y^2}dx\wedge dy$.  It is well-known that the isometry group of $\mathbb{H}$ is $PSL(2,\mathbb{R})= SL(2,\mathbb{R})/\lbrace\pm 1\rbrace$, whose elements act as M{\"o}bius transformations $z\mapsto\frac{az+b}{cz+d}$, where $a,b,c,d\in\mathbb{R}$ and $ad-bc=1$.  The action of $PSL(2,\mathbb{R})$ on $\mathbb{H}$ is transitive because the M{\"o}bius transformation represented by $\frac{1}{\sqrt{y}}\begin{psmallmatrix} y & x \\ 0 & 1\end{psmallmatrix}$ takes $i\in\mathbb{H}$ to $x+iy$.  It is easy to see that the stabilizer of $i$ is
\[
\operatorname{stab}(i) = \left.\bigg\lbrace\varphi_\alpha :=\begin{psmallmatrix} \cos(\alpha/2) & \sin(\alpha/2) \\ -\sin(\alpha/2) & \cos(\alpha/2)\end{psmallmatrix}\bigg |\,\alpha\in[0,2\pi)\bigg\rbrace\right/\{\pm 1\}\cong S^1
\]
In fact, the derivative of $\varphi_\alpha: \hh\to\hh$ at $z = i$ is a counterclockwise rotation by $\alpha$.
Since an isometry of $\hh$ is determined uniquely by its value and derivative at a single point, it follows easily that $PSL(2,\arr)$ acts simply transitively on the unit tangent bundle $T^1\hh$. Indeed, we get an orbit map 
\begin{eqnarray*}
&\Phi: PSL(2, \arr)\to T^1\hh \cong \hh\times S^1& \\
&\Phi(g) = (g(i), dg_i(1)) \leftrightarrow \left(\frac{ai + b}{ci+d}, -2\arg(ci + d)\right)&
\end{eqnarray*}
where we think of $1\in T_i\hh = \cee$ as a unit tangent vector via the canonical trivialization, and identify the unit vector $dg_i(1)\in T_{g(i)}\hh$ with its angular coordinate. The orbit map is a diffeomorphism that we can describe explicitly as follows. If $g = \begin{psmallmatrix} a & b \\ c & d\end{psmallmatrix}$ is an arbitrary M{\"o}bius transformation with $g(i) = z = x+iy$, then $g$ differs from the map $g_z = \frac{1}{\sqrt{y}}\begin{psmallmatrix} y & x \\ 0 & 1\end{psmallmatrix}$ by composition with a rotation $\varphi_\alpha$ as above: $g = g_z\cdot \varphi_\alpha$ for some unique $\alpha$. Thus in $PSL(2,\arr)$,
\begin{eqnarray}
\nonumber\begin{psmallmatrix} a & b \\ c & d\end{psmallmatrix} &=& \ts\frac{1}{\sqrt{y}}\begin{psmallmatrix} y & x \\ 0 & 1\end{psmallmatrix}\begin{psmallmatrix} \cos(\frac{\alpha}{2}) & \sin(\frac{\alpha}{2}) \\ -\sin(\frac{\alpha}{2}) & \cos(\frac{\alpha}{2})\end{psmallmatrix}\\
\label{actionformula}& =& \ts\frac{1}{\sqrt{y}}\begin{psmallmatrix} y\cos(\frac{\alpha}{2})-x\sin(\frac{\alpha}{2}) & y\sin(\frac{\alpha}{2})+x\cos(\frac{\alpha}{2}) \\ -\sin(\frac{\alpha}{2}) & \cos(\frac{\alpha}{2})\end{psmallmatrix}
\end{eqnarray}
for some $\alpha\in[0,2\pi)$. In particular $\begin{psmallmatrix} a & b \\ c & d\end{psmallmatrix}\in PSL(2,\arr)$ is determined uniquely by $(x+iy, \alpha)\in \hh\times S^1$, verifying invertibility of $\Phi$. Writing $x,y,\alpha$ in terms of $a,b,c,d$ gives:
\[
\Phi\begin{psmallmatrix} a & b \\ c & d\end{psmallmatrix} = \ts\left( \frac{ac+bd}{c^2+d^2},\frac{1}{c^2+d^2}, -2\arg(ci + d)\right)\in \hh\times S^1.
\]
More generally, we have:

\begin{lemma}\label{actionlemma} Given $h = \begin{psmallmatrix} a & b \\ c & d\end{psmallmatrix}\in PSL(2,\arr)$, the left action of $h$ on $PSL(2,\arr)$ is identified under $\Phi$ with the map $\hh\times S^1\to \hh\times S^1$ given by
\[
h: (w, \theta) \mapsto \left(\frac{aw+b}{cw+d}, \theta - 2\arg(cw + d)\right)
\]
\end{lemma}

\begin{proof} If we write $p = (w, \theta)= \Phi(g)\in \hh\times S^1$ then $p = (g(i), (dg)_i(1))$  where $g = g_w\cdot \varphi_\theta$ as above. Since, under the canonical trivialization of $T\hh$, $(dg_w)_i$ is the identity and $(d\varphi_\alpha)_i = e^{i\theta}$ (as a linear map $T_i\hh\to T_i\hh$), we have $(dg)_i(1) = e^{i\theta}\in T_w\hh = \cee$. Therefore
\[
h(p) = (h(w), dh_w(e^{i\theta})) = \ts\left(h(w), \frac{e^{i\theta}}{(cw + d)^2}\right),
\]
where the second coordinate describes a unit vector in $T_{h(w)}\hh$, whose angular coordinate is $\theta - 2\arg(cw+ d)$.
\end{proof} 

To describe the universal cover $\SLt$ of $PSL(2,\arr)$, observe that the discussion above implies that $\Phi$ lifts to a diffeomorphism $\tilde\Phi: \SLt\to \hh\times \arr$, where the latter is the universal cover of $T^1\hh$. Strictly, there is a choice involved in defining $\tilde\Phi$, so we determine it uniquely by specifying $\tilde\Phi(e) = (i, 0)\in \hh\times \arr$ where $e\in \SLt$ is the identity element.

\begin{theorem}\label{invcontact} Let $\hh\times \arr$ be the universal cover of $\hh\times S^1$.  Then the 1-form
\[
\ts\lambda = dt + \frac{1}{y}dx
\]
is invariant under the action of $\SLt$ corresponding under $\tilde\Phi$ to left multiplication, where $(x,y,t)$ are the usual coordinates on $\hh\times \arr$. Moreover, $\lambda$ is a contact form whose Reeb vector field is $R_\lambda = \partial_t$.
\end{theorem}

Note that the existence of left-invariant contact forms on $\SLt \cong \hh\times \arr$ is well-known. Our aim is to make this explicit enough to allow the perturbations and calculations in the following sections.

\begin{proof} Observe that there is a short exact sequence of Lie groups corresponding to the universal cover:
\begin{equation}\label{ses}
1\to \zee\to \SLt \to PSL(2,\arr)\to 1,
\end{equation}
describing $\SLt$ as a central extension. The action of a generator of $\zee$ by left multiplication in $\SLt$ corresponds to the covering transformation, and under $\tilde\Phi$ this can be taken to be the vertical shift $t\mapsto t+2\pi$. Clearly $\lambda$ is invariant under this action, and from this it follows that to prove $\SLt$ invariance of $\lambda$ it suffices to prove $PSL(2,\arr)$ invariance of the corresponding form $\lambda_0 = d\theta + \frac{1}{y}dx$  on $\hh\times S^1$. For this it is convenient to use the complex coordinate $z = x+iy$ on $\hh$, so that $\lambda_0 = d\theta + \frac{i}{z - \bar z}(dz + d\bar z)$. Now explicit computation gives that for $g = \begin{psmallmatrix} a & b \\ c & d\end{psmallmatrix}$,
\[
g^*(\lambda_0) = g^*d\theta + \frac{i}{z + \bar z}\left(\frac{c\bar z + d}{cz + d}\, dz + \frac{cz + d}{c\bar z + d}\,d\bar z\right).
\]
(Note that $\theta$ here is the coordinate on the $S^1$ factor of $\hh\times S^1\cong T^1\hh$, on which the M\"obius transformation $g$ acts via its derivative.) Moreover, observe that for a smooth map $f$ between subsets of $\cee$ we have $f^*d\theta = d(\arg(f)) = -\frac{i}{2f} df +\frac{i}{2\bar f}d\bar f$. It follows that 
\begin{equation}\label{thetapullback}
g^*d\theta = d(\theta - 2\arg(cz +d)) = d\theta + \frac{ic}{cz+d}\,dz - \frac{ic}{c\bar z + d}\,d\bar z.
\end{equation}
Combining the two equations above with a little algebra gives the desired invariance. That $\lambda$ is a contact form with Reeb field $\partial_t$ is a simple exercise.
\end{proof}

\begin{proposition}\label{trivprop} The contact structure $\xi =\ker\lambda$ on $\hh\times \arr$ admits a global, $\SLt$-invariant trivialization given by 
\[
e_1 =  \left[\begin{array}{c} y\cos t \\ y\sin t\\ -\cos t \end{array}\right] \quad e_2 =\left[\begin{array}{c} -y\sin t \\ y\cos t\\ \sin t \end{array}\right].
\]
\end{proposition}
\begin{proof} We have a basis for $\xi_{(0,1,0)} = \ker(dt + dx)$ given by $\{\partial_x - \partial_t, \partial_y\}$, which we transport to all of $\hh\times\arr$ by the derivative of the $\SLt$ action. To make the calculation it suffices to replace $t$ by $\theta$ (i.e., work with the $PSL(2,\arr)$ action on $\hh\times S^1$), by invariance under $2\pi\zee$ translation. 

In terms of the coordinates $(z,\theta)$, the derivative of the action of $g = \begin{psmallmatrix} a & b \\ c & d\end{psmallmatrix}$ at $(z,\theta) = (i,0)$ is given by
\[
dg_{(i,0)} = \left[\begin{array}{c} \frac{dz}{(ci + d)^2}\\ d\theta + 2 Re\left(\frac{ic\,dz}{ci + d}\right)\end{array}\right]
\]
(differentiate the formula of Lemma \ref{actionlemma} at $(w,\theta) = (i,0)$, and use \eqref{thetapullback}).
Evaluating on $\partial_x - \partial_\theta = \partial_z + \partial_{\bar z} - \partial_\theta$ and $\partial_y = i(\partial_z - \partial_{\bar z})$ gives, respectively,
\[
\left[\begin{array}{c} \frac{1}{(ci+d)^2}\\ -1+\frac{2c^2}{c^2 + d^2}\end{array}\right] \quad\mbox{and}\quad \left[\begin{array}{c} \frac{i}{(ci+d)^2} \\ -\frac{2cd}{c^2+d^2}\end{array}\right].
\]
Passing back to real coordinates and using \eqref{actionformula} to write this in terms of $(x,y,\theta)$ gives the result.
\end{proof}

\begin{remark} The contact form $\lambda$ is positive, if $\hh\times \arr$ is oriented as usual by $\{\partial_x, \partial_y, \partial_t\}$. Likewise, $d\lambda$ determines an orientation on $\xi$, with respect to which the basis above is positive.
\end{remark}

\begin{remark} It is known that we can equip $\SLt$ with a left-invariant Riemannian metric for which the contact structure $\xi$ is just the orthogonal complement of the $t$-axis. Moreover, $\xi$ is invariant under all orientation-preserving isometries of this metric. The form $\lambda$ is invariant under the identity component of the isometry group, which is generated by left multiplications together with arbitrary shifts in the $t$-direction (and is sent to $-\lambda$ under the representative $(x,y,t)\mapsto (-x,y,-t)$ of the other component of orientation-preserving isometries). The trivialization of Proposition \ref{trivprop}, however, is preserved only by the group $\SLt\subset \isom(\SLt)$.
\end{remark}

In the discussion to follow, it will be useful to think of the lifted orbit map
\[
\tilde\Phi:\SLt\to \widetilde{T^1\hh}\cong\hh\times\arr 
\]
 as mapping into $\hh\times \widetilde{\cee^\times}$, where $\cee^\times = \cee - \{0\}$, and $\widetilde{\cee^\times}\to \cee^\times$ is the universal cover. Specifically, we think of the unit tangent bundle $T^1\hh\cong \hh\times S^1$ as lying in the deleted tangent bundle $T^\times \hh = \hh\times \cee^\times$ and then pass to the universal cover. From this point of view, it was observed by Milnor that $\SLt$ can be identified with the group of ``labeled biholomorphic maps'' from $\hh$ to $\hh$.

\begin{definition} A {\em labeled holomorphic map} between open sets $U, V\subset \cee$ is a holomorphic map $g:U\rightarrow V$ such that $g'(z)$ is nowhere-vanishing, together with a continuous lift $\widetilde{g'}:U\rightarrow\widetilde{\mathbb{C}^\times}$ of its derivative.  Such a lift is called a {\em label}. 
\end{definition}

Note that there is an isomorphism of abelian Lie groups $\tilde{e}: \cee\to \widetilde{\cee^\times}$ lifting the exponential map $\cee\to \cee^\times$. In particular we write the group operation in $\widetilde{\cee^\times}$ multiplicatively, corresponding to addition in $\cee$. The projection $\widetilde{\cee^\times}\to \cee^\times$ will be written $p$, and our convention is that $p\circ \tilde{e}(z) = e^{iz}$:
\[
\xymatrix{
\cee \ar[r]^{\tilde e}\ar[rd]_{e^{iz}} & \wt{\cee^\times}\ar[d]^p\\
&\cee^\times}
\]

Labeled holomorphic maps can be composed: indeed, given maps $g_0: U_0\to U_1$ and $g_1: U_1\to U_2$ with labels, there is a unique labeling for $g_1\circ g_0$ satisfying the chain rule: $\widetilde{(g_1\circ g_0)'}(z) = \widetilde{g_1'}(g_0(z))\cdot \widetilde{g_0'}(z)$. With this understood, it is not hard to check that the group of labeled biholomorphic maps $\hh\to \hh$ is naturally isomorphic to $\SLt$. 

There is a natural action of a labeled biholomorphic map on $\hh\times \wt{\cee^\times}$: namely, if $\ul{g} = (g, \wt{g'})$ is a labeled biholomorphic map on $\hh$, then we set
\[
\ul{g}(z,u) = (g(z), \wt{g'}(z)\cdot u)
\]
for $(z,u)\in \hh\times\wt{\cee^\times}$. Since a biholomorphic map $\hh\to\hh$ is the same as a hyperbolic isometry, the construction means that the left action of $\SLt$ preserves the lifted unit tangent bundle $\widetilde{T^1\hh}\subset \hh \times \wt{\cee^\times}$, and this action is the same as the one studied in the proof of Theorem \ref{invcontact}. 

Observe that since a unit tangent vector at a point $z\in\hh$ is of the form $(z, (\im z)e^{i t})$, we can describe the lifted unit tangent bundle as the image of the set
\begin{equation}\label{liftedunit}
\{(z, t - i\log(\im z))\,|\, z\in \hh,\, t\in \arr\}\subset \hh\times \cee
\end{equation}
under the isomorphism $1\times \tilde e : \hh\times\cee\to \hh\times \wt{\cee^\times}$.

Modifying Milnor's terminology slightly, we define a {\it (holomorphic) differential of degree $k$} on an open set $U\subset \cee$ to be a function $\phi:U\times\mathbb{C}^{\times}\rightarrow\mathbb{C}$ of the form $\phi(z,u) = f(z)u^k$, where $f:U\rightarrow\mathbb{C}$ is holomorphic. Here $k$ must be an integer, since rational powers of complex numbers are multi-valued. Since we wish to consider forms of fractional degree as well, we pass to the universal cover. So for a rational number $a\in\cue$, we define a \textit{differential of degree $a$ on $U$} to be a function $\phi:U\times\wt{\cee^\times}\to\cee$ satisfying $\phi(z,u) = f(z)p(u^a)$, where $f:U\rightarrow\mathbb{C}$ is again holomorphic. Equivalently, if $u = \tilde{e}(w)$ for $w\in\cee$, we can write $\phi(z,w) = f(z)\,p(\tilde{e}(aw)) = f(z) e^{iaw}$.

Labeled holomorphic maps act on differentials by pullback: namely, for $\ul g$ labeled holomorphic and $\phi$ a differential of degree $a$, we define 
\[
\ul g^*\phi(z, u) = \phi(g(z), \wt{g'}(z)u).
\]
A differential $\phi$ will be said to be {\it $\tG$-automorphic}, for a subgroup $\tG< \SLt$, if $\ul g^*\phi = \phi$ for all $\ul g\in \tG$.

We will be interested in compact 3-manifolds that arise as the quotient of $\SLt$ by such a subgroup $\tG$, specifically those $\tG$ that act on $\hh$ cocompactly. Let us rewrite the exact sequence \eqref{ses} as
\[
1\to C \to \SLt\to \isom^+(\hh)\to 1,
\]
where $\isom^+(\hh)\cong PSL(2,\arr)$ indicates the orientation-preserving isometries and $C\cong \zee$ is the center of $\SLt$. Then $\tG< \SLt$ acts on $\hh$, factoring through the quotient $\tG/(\tG\cap C)$. For use below, we recall:

\begin{lemma}\label{Auto0} Suppose that $\widetilde{G}<\widetilde{SL(2,\mathbb{R})}$ is such that $G = \widetilde{G}/(\widetilde{G}\cap C)<PSL(2,\mathbb{R})$ acts on $\mathbb{H}$ with compact fundamental domain.  Then we have the following:
\begin{description}
\item[(i)] (\cite{milnor1975}, Lemma 5.4) $\widetilde{G}$ is a discrete subgroup of $\widetilde{SL(2,\mathbb{R})}$ and the space of right cosets $\wt G \bs \widetilde{SL(2,\mathbb{R})}$ is compact.  Also $\widetilde{G}\cap C$ is necessarily nontrivial.

\item[(ii)] (\cite{milnor1975}, Theorem 5.9) Two points $(z,w),(z',w')\in\mathbb{H}\times\mathbb{C}$ belong to the same $\widetilde{G}$-orbit iff $\phi(z,w) = \phi(z',w')$ for every $\widetilde{G}$-automorphic form $\phi$.
\end{description}
\end{lemma}

We also have the following.

\begin{lemma}\label{invlemma} Let $\wt{G}<\wt{SL(2,\arr)}$ be a subgroup as in the previous lemma, and suppose $\phi: \hh\times\wt{\cee^\times}\to \cee$ is $\wt{G}$-automorphic. Then the function
\[
(z,t)\mapsto |\phi\circ j(z,t)|: \hh\times\arr \xrightarrow{j} \wt{T^1\hh}\subset\mathbb{H}\times\widetilde{\cee^\times}\xrightarrow{|\phi|}\arr
\]
is $\wt{G}$-invariant and independent of $t$, hence descends to a function $G\bs\hh\to \arr$.
\end{lemma}

Here we write $G$ for $\widetilde{G}/(\widetilde{G}\cap C)$, and $j$ for the composition 
\[
(z,t) \mapsto (z, \tilde{e}(t - i\log(\im z))).
\]

\begin{proof} Write $\phi(z,w) = f(z) e^{iaw}$ for $(z,w)\in \hh\times \cee$ as above. Then from \eqref{liftedunit} we have
\[
\phi\circ j(z,t) = f(z) (\im z)^a\,e^{iat},
\]
whose absolute value depends only on $z$. The invariance is clear.

\end{proof}

\subsection{Triangle Groups and Central Extensions}\label{Bri2}

We now describe the construction of a particular family of groups $\Pi(p,q,r)< \SLt$, and sketch the proof that the quotient of $\SLt$ by the left action of $\Pi(p,q,r)$ is diffeomorphic to the Brieskorn sphere $\Sigma(p,q,r)$. Throughout, we suppose $p,q,r$ are relatively prime positive integers such that $\frac{1}{p} +\frac{1}{q} + \frac{1}{r} < 1$. This condition means that there is a (unique up to isometry) hyperbolic triangle in $\mathbb{H}$ with interior angles $\pi/p$, $\pi/q$, and $\pi/r$, which we denote by $T(p,q,r)$.  We will write $v_1$, $v_2$, and $v_3$ for the vertices of $T(p,q,r)$ adjacent to the interior angles $\pi/p$, $\pi/q$, and $\pi/r$, respectively, and assume these are arranged in counterclockwise cyclic order.  Also label the edges $e_1$, $e_2$, and $e_3$ of $T(p,q,r)$ so that $e_j$ is opposite to the vertex $v_j$. For $j= 1,2,3$, let $\sigma_j$ be the hyperbolic reflection of $\hh$ across the geodesic containing the edge $e_j$. The (full) Schwarz triangle group $\Delta^*(p,q,r)<\isom(\hh)$ is the group of hyperbolic isometries generated by $\sigma_1, \sigma_2$ and $\sigma_3$.  We will be primarily concerned with the index $2$ subgroup $\Delta(p,q,r)<\Delta^*(p,q,r)$ consisting of orientation-preserving isometries, which we will simply call the \textit{triangle group}.  The full Schwarz group $\Delta^*(p,q,r)$ has the presentation 
\[
\Delta^*(p,q,r) = \langle \sigma_1,\sigma_2,\sigma_3\,|\,\sigma_j^2 = 1, (\sigma_2\sigma_3)^p = (\sigma_3\sigma_1)^q = (\sigma_1\sigma_2)^r = 1\rangle,
\]
as shown in \cite{milnor1975}, for example.  Observe that, for instance, $\sigma_1\sigma_2$ is a reflection first in $e_2$ and then in $e_1$, which is seen to be a rotation by $\frac{2\pi}{r}$ radians about $v_3$.  It can be shown using the Reidemeister-Schreier theorem that the orientation-preserving subgroup has a presentation 
\[
\Delta(p,q,r) = \langle\delta_1,\delta_2,\delta_3\,|\,\delta_1^p = \delta_2^q = \delta_3^r = \delta_1\delta_2\delta_3 = 1\rangle,
\]
where $\delta_j = \sigma_k\sigma_{\ell}$, for $\{j,k,\ell\}$  a cyclic permutation of $\{1,2,3\}$.

The \textit{centrally extended triangle group} $\widetilde{\Delta}(p,q,r)$ is defined to be the preimage of $\Delta(p,q,r)$ in $\SLt$ under the covering projection $\SLt\to PSL(2,\mathbb{R})$.  Hence we see that there is a short exact sequence
\[
1\rightarrow{ C}\rightarrow\widetilde{\Delta}(p,q,r)\rightarrow\Delta(p,q,r)\rightarrow1
\]
where $C\cong\mathbb{Z}$ embeds as the center of $\wt\Delta(p,q,r)$.  There is also a group presentation 
\[
\widetilde{\Delta}(p,q,r) = \langle g_1,g_2,g_3\, |\, g_1^p = g_2^q = g_3^r = g_1g_2g_3 = c\rangle,
\]
 where $c$ generates $C$.  Finally define 
 \[
 \Pi(p,q,r) = [\widetilde{\Delta}(p,q,r),\widetilde{\Delta}(p,q,r)]. 
 \]
 By \cite[Corollary 3.2]{milnor1975}, $\Pi(p,q,r)$ has index $d=\frac{pqr}{s}$ in $\widetilde{\Delta}(p,q,r)$, where 
 \[
\ts \frac{1}{s} = 1-\frac{1}{p}-\frac{1}{q}-\frac{1}{r},\quad\mbox{so that}\quad d = pqr - qr-pr-pq.
 \]
 It can be shown that $ c^d$ generates the center of $\Pi(p,q,r)$, which is $\Pi(p,q,r)\cap C$, and that when $p,q,r$ are relatively prime, $\Pi(p,q,r)$ maps onto $\Delta(p,q,r)$. Thus we obtain another short exact sequence
\begin{equation}\label{pises}
1\rightarrow C_0\rightarrow\Pi(p,q,r)\rightarrow\Delta(p,q,r)\rightarrow1
\end{equation}
where $C_0<C$ is an infinite cyclic subgroup of index $d$.
If we take as known (for the moment) the fact that the Brieskorn sphere $\Sigma(p,q,r)$ is the quotient of $\SLt$ by the left action of $\Pi(p,q,r)$, then we can recognize this short exact sequence as the one in \cite[Lemma 3.2]{scott83}, where $\Pi(p,q,r) = \pi_1(\Sigma(p,q,r))$ and $\Delta(p,q,r)$ is the orbifold fundamental group of $\mathbb{H}/\Delta(p,q,r)$.  In effect this says that $c^d$ represents a regular fiber of the Seifert structure on $\Sigma(p,q,r)$ (see below), which is a generator of the infinite cyclic group $C_0$ appearing in the short exact sequence.  Indeed we can also use \cite[Corollary 3.2]{milnor1975} to show that $\Pi(p,q,r)$ contains the elements $g_1^d$, $g_2^d$, and $g_3^d$.  This shows that $(g_1^d)^p = (g_1^p)^d = c^d = (g_2^d)^q = (g_3^d)^r$.

The fact that $\Sigma(p,q,r)\cong\Pi(p,q,r)\bs\SLt$ hinges on the following lemma.

\begin{lemma}\label{Auto1} (\cite{milnor1975}) Let $p,q,r\geq 2$ and $\frac{1}{p}+\frac{1}{q}+\frac{1}{r}<1$.  There exist three $\Pi(p,q,r)$-automorphic forms on $\mathbb{H}$, $\phi_1$, $\phi_2$, and $\phi_3$ of degrees $s/p$, $s/q$, and $s/r$, respectively such that $\phi_i$ has a simple zero at each point in the orbit ${\Delta}(p,q,r)v_i$, and the $\phi_i$ generate the algebra of all $\Pi(p,q,r)$-automorphic forms.  Furthermore the $\phi_i$ satisfy the polynomial relation $\phi_1^p+\phi_2^q+\phi_3^r = 0$.
\end{lemma}

%
%

It follows from the lemma that  the mapping $(\phi_1,\phi_2,\phi_3):\hh\times\wt{\cee^\times}\rightarrow\mathbb{C}^3$ has image contained in the {\it Brieskorn variety} defined by
\[
V(p,q,r) = \{x^p + y ^q + z^r = 0\}\subset \cee^3,
\]
and is invariant under the action of $\Pi(p,q,r)$.  Thus we get a well-defined holomorphic map 
\[
F: \Pi(p,q,r)\bs \hh\times\wt{\cee^\times} \to V(p,q,r) \bs\{0\}.
\]
That $F$ is one-to-one is a consequence of the second part of Lemma \ref{Auto0}; Milnor verifies it is biholomorphic. 

Now, we can map $\SLt$ into $V\bs\lbrace0\rbrace$ via the composition
\[
\SLt\cong\mathbb{H}\times\mathbb{R}\xrightarrow{\tilde\Phi}\mathbb{H}\times\widetilde{\cee^\times}\xrightarrow{F}V\bs\lbrace0\rbrace.
\]
The image of this map need not lie in $\Sigma(p,q,r)\subset V(p,q,r) - \{0\}$, but is diffeomorphic to it, as follows.
Let $\ul g = (g, \wt{g'})\in\SLt$, considered as a labeled holomorphic map $\hh\to\hh$ and identified with the point 
\[
\tilde{\Phi}(\ul g) = (g(i), \wt{g'}(i))= \ul g(i, 1) \in \hh\times \wt{\cee^\times},
\]
 and write $F([g(i,1)]) = (z_1,z_2,z_3)\in V\bs\lbrace0\rbrace$, where $[g(z,1)]$ indicates the orbit under the left action of $\Pi(p,q,r)$.  It is easy to see that the curve 
\[
t\mapsto F([\ul g(i, t^{1/s})]) = (t^{1/p}z_1,t^{1/q}z_2,t^{1/r}z_3), 
\]
for $t>0$, intersects $S^5\subset\mathbb{C}^3$ transversely and precisely once.  So we can define 
 \[
 \Psi:\Pi(p,q,r)\bs\SLt\rightarrow\Sigma(p,q,r) = V(p,q,r)\cap S^5
 \]
  by letting $\Psi(\Pi(p,q,r)\ul g)$ be this intersection point.  One verifies that $\Psi$ is smooth, injective, and that $d\Psi$ has maximal rank everywhere.  Since $\Pi(p,q,r)\bs\SLt$ is compact and $\Sigma(p,q,r)$ is connected it now follows from the inverse function theorem that $\Psi$ is a diffeomorphism.

This relates to another classical description of $\Sigma(p,q,r)$ as a Seifert fiber space.   Observe that there is a fixed point free action of $S^1$ on $V(p,q,r)$ given by $\zeta\cdot(z_1,z_2,z_3) = (\zeta^{qr}z_1,\zeta^{pr}z_2,\zeta^{pq}z_3)$ where $\zeta$ is a unit complex number, which restricts to an $S^1$ action on $\Sigma(p,q,r)$. Moreover, the diffeomorphism constructed above between $\Pi(p,q,r)\bs \hh\times \arr$ and $\Sigma(p,q,r)$ is equivariant, in the sense that 
\[
F\circ\tilde\Phi(z,t_0 + t) = e^{it/d}\cdot F\circ\tilde\Phi(z, t_0)
\]
(the diffeomorphism between the image of $F\circ \tilde\Phi$ and $\Sigma(p,q,r)$ clearly commutes with the $S^1$ action).

Hence we have an identification between $\Sigma(p,q,r)$ as a Seifert space and the quotient $\Pi(p,q,r)\bs \hh\times\arr$, where the action of $S^1$ on the latter is induced by $\arr$ translation. In our parametrization the action has period $2\pi d$, and dividing by the action gives a projection
\[
\Pi(p,q,r)\bs \hh\times\arr\to \Delta(p,q,r)\bs \hh \cong S^2(p,q,r)
\]
(c.f. \eqref{pises}). Here $S^2(p,q,r)$ denotes the orbifold with underlying space $S^2$ and three cone points of orders $p$, $q$, and $r$ corresponding to the vertices of the initial triangle. The circle fibers over these points, consisting of points with nontrivial stabilizer in $S^1 = \arr/ 2\pi d\zee$, have period $2\pi d/p$, $2\pi d/q$, and $2\pi d/r$ respectively. In terms of the fundamental group these exceptional fibers are conjugate to the generators $g_i^d$.

\subsection{Contact Structures}\label{contstrident}

In section \ref{Bri1} we described a contact form $\lambda$ on $\hh\times \arr$ invariant under the left action of $\SLt$. In particular, fixing relatively prime integers $p,q,r$ as above, the quotient $Y = \Pi(p,q,r)\bs \hh\times\arr$ inherits a contact form that we also denote by $\lambda$. On the other hand the description of the Brieskorn sphere $\Sigma(p,q,r)$ as the link of the singularity in $V(p,q,r)$ also determines a contact structure, as the field of complex tangencies. The latter is also known as the {\it Milnor fillable} contact structure on $\Sigma(p,q,r)$. In this section we show that these two contact structures correspond under the diffeomorphism $\Psi$ constructed above. 

Indeed, observe that the Reeb vector field for $\lambda$ is (the image of) $\partial_t$, which is the infinitesimal generator of the circle action on $Y$. On the other hand, a contact form for the natural contact structure on $\Sigma(p,q,r)$ is the restriction of the form 
\[
\alpha_{p,q,r} = \frac{i}{4} (p(z_1d\bar{z}_1-\bar{z}_1 dz_1)+q(z_2d\bar{z}_2-\bar{z}_2 dz_2)+r(z_3d\bar{z}_3-\bar{z}_3 dz_3))
\]
from $\cee^3$ to $\Sigma(p,q,r)$ (see \cite[Section 7.1]{geigesbook}, for example), and the corresponding Reeb vector field is
\[
R_{\alpha_{p,q,r}}  = 2i\bigg(\frac{1}{p}(z_1\partial_{z_1}-\bar{z}_1 \partial_{\bar{z}_1})+\frac{1}{q}(z_2\partial_{z_2}-\bar{z}_2\partial_{\bar{z}_2})+\frac{1}{r}(z_3\partial_{z_3}-\bar{z}_3\partial_{\bar{z}_3})\bigg)
\]
This vector field generates the flow
\[
\varphi_t(z_1,z_2,z_3) = (e^{2it/p}z_1, e^{2it/q}z_2, e^{2it/r}z_3)
\]
 which, up to a constant reparametrization, is the same as the $S^1$ action on $\Sigma(p,q,r)$ described above. 
 
It follows that the $S^1$-equivariant diffeomorphism $\Psi: Y\to \Sigma(p,q,r)$ pulls back $\alpha_{p,q,r}$ to a contact form on $Y$ whose Reeb vector field is a (constant) multiple of the Reeb field of $\lambda$. From this it is easy to see that the family $(1-t)\lambda + t\Psi^*\alpha_{p,q,r}$ is a path of contact forms and hence by Gray stability the corresponding contact structures are isotopic.

\begin{corollary} The Milnor fillable contact structure on $\Sigma(p,q,r)$ (for $\frac{1}{p} + \frac{1}{q} + \frac{1}{r} < 1$) is hypertight.
\end{corollary}

\begin{proof} We have just seen that the Reeb orbits on $\Sigma(p,q,r)$ lift to orbits of $\partial_t$ on $\hh\times \arr$. In particular no Reeb orbit lifts to a loop in the universal cover $\hh\times \arr$, so none are contractible.  
\end{proof}

\section{Non-degenerate Contact Form and Conley-Zehnder Indices}\label{perturbsec}

In light of the results from the previous section, we will now blur the distinction between the Brieskorn sphere $\Sigma(p,q,r)$ (the link of the singularity $V(p,q,r)$) and the quotient $Y = \Pi(p,q,r)\bs\SLt$. In particular, we use the contact form $\lambda=\lambda_{p,q,r}$ induced from $dt + \frac{1}{y} dx$ on $\hh\times \arr \cong \SLt$ as in Theorem \ref{invcontact}. As we have seen, the Reeb vector field of $\lambda_{p,q,r}$ generates the circle action in the Seifert fibration $\Sigma(p,q,r)\rightarrow S^2(p,q,r)$.  Consequently, $\lambda_{p,q,r}$ has no isolated Reeb orbits.  In this section we define a perturbation of $\lambda_{p,q,r}$ to give a nondegenerate contact form on $\Sigma(p,q,r)$, which we can then use to compute contact homology.  Strictly, our construction gives a contact form that is nondegenerate ``up to large action'' in the sense of Nelson \cite{Ne18}, but that will be sufficient for our purposes.

Let $\pi:\Sigma(p,q,r)\rightarrow S^2(p,q,r)$ denote the Seifert fibration, where as before $S^2(p,q,r) = \Delta(p,q,r)\bs\hh$ is the 2-sphere with three orbifold singularities of orders $p,q$ and $r$. To make certain arguments more convenient later on, recall that a classical result of Fox \cite{fox52} (see also \cite{selberg}) implies that the triangle group $\Delta(p,q,r)$ contains a subgroup of finite index that acts freely on $\hh$. If $G(p,q,r)$ is such a subgroup, this means that $G(p,q,r)\bs\hh$ is a smooth closed surface that is a finite (orbifold) cover of $S^2(p,q,r)$. Letting $\wt G(p,q,r) < \Pi(p,q,r)$ be the lifted subgroup \eqref{pises}, it follows that $\wt G(p,q,r)\bs (\hh\times \arr) \to G(p,q,r)\bs\hh$ is a smooth principal $S^1$ bundle. By way of notation we write $E(p,q,r) = \wt G(p,q,r)\bs(\hh\times \arr)$ and $B(p,q,r) = G(p,q,r)\bs\hh$, so there is a commutative diagram
\[
\xymatrix{
E(p,q,r) \ar[r]^\EtoSg\ar[d]_\EtoB & \Sigma(p,q,r) \ar[d]^{\SgtoS2}\\
B(p,q,r)\ar[r]_{\BtoS2} & S^2(p,q,r)
}
\]
in which the top arrow is a map of $S^1$ spaces and a finite sheeted covering map, the vertical arrows are quotients by $S^1$ actions, and the bottom arrow is a finite orbifold covering (the choice of the group $G(p,q,r)$ will not affect our arguments).

  We will define a function $H:S^2(p,q,r)\rightarrow\mathbb{R}$ that is smooth away from the singular points $[v_i]$, such that $H\circ\BtoS2$ is a Morse function and such that $(1+\epsilon (H\circ\SgtoS2)^*)\lambda_{p,q,r}$ is nondegenerate for small $\epsilon$.  We will construct $H$ in stages, first defining it near the singular locus of $S^2(p,q,r)$, then extending it as a Morse function on the rest of $S^2(p,q,r)$.  The form $(1+\epsilon (H\circ\SgtoS2)^*)\lambda_{p,q,r}$ will have isolated Reeb orbits that project to the points $[v_i]\in S^2(p,q,r)$.  Equivalently, these Reeb orbits correspond to the singular $S^1$-fibers in $\Sigma(p,q,r)$.  There will be some other non-degenerate Reeb orbits of the perturbed contact form also, corresponding to critical points of $H$, but since $H\circ\BtoS2$ is Morse we can analyze these using techniques from \cite{Ne18}.  For concreteness, we will begin by working over the point $[v_1]$ corresponding to the vertex of $T(p,q,r)$ adjacent to the interior angle $\frac{\pi}{p}$.

\subsection{Local Perturbations}\label{Non1}
Our strategy to produce a function $H: S^2(p,q,r)\to \arr$ as above is essentially to look for functions $\hh\times\arr\to \arr$ that are invariant both under the action of $\Pi(p,q,r)$ and under all vertical translations. If $f$ is any such invariant function and is positive, then $f\lambda$ is still an invariant contact form defining the same contact structure. 

For any positive real-valued function $f$ on $\hh\times\arr$ it is not hard to compute the Reeb vector field of $f\lambda$.  Namely, we have  $R_{f\lambda} = \frac{1}{f}R_{\lambda}+V_\frac{1}{f}$ where the ``contact Hamiltonian vector field'' $V_\frac{1}{f}$ is the unique vector field satisfying (i): $V_\frac{1}{f}\in\ker\lambda$ and  (ii): $d\lambda(V_\frac{1}{f},\cdot) = d(\frac{1}{f})(R_{\lambda})\lambda-d(\frac{1}{f})$.  One easily checks that for our standard $\lambda$, 
\[ 
V_\frac{1}{f} = \begin{pmatrix} -y^2\partial_y\frac{1}{f} \\ -y\partial_{t}\frac{1}{f}+y^2\partial_x\frac{1}{f} \\ y\partial_y\frac{1}{f}\end{pmatrix}
\]
from which it follow for any positive function $f$,
\[
R_{f\lambda} = \begin{pmatrix} -y^2\partial_y\frac{1}{f} \\ -y\partial_{t}\frac{1}{f}+y^2\partial_x\frac{1}{f}\\ \frac{1}{f}+y\partial_y\frac{1}{f}\end{pmatrix}
\]

Now consider the function $f = 1+\epsilon|\phi_1|^2$ defined on $\mathbb{H}\times\mathbb{R}$, where $\phi_1$ is the $\Pi(p,q,r)$-automorphic form with a simple zero at $[v_1]$ described in section \ref{Bri1}, for some small $\epsilon>0$ (in particular, where convenient we assume $\epsilon \in (0,1)$).  From Lemma \ref{invlemma} we know $|\phi_1|$ restricts to a function on $\mathbb{H}\times\mathbb{R}\cong\SLt$ that is independent of $t\in\arr$.  Therefore we can write
\[
R_{f\lambda} = \begin{pmatrix} y^2\frac{\epsilon\partial_y|\phi_1|^2}{(1+\epsilon|\phi_1|^2)^2} \\ -y^2\frac{\epsilon\partial_x|\phi_1|^2}{(1+\epsilon|\phi_1|^2)^2} \\ \frac{1}{1+\epsilon|\phi_1|^2}-\frac{\epsilon\partial_y|\phi_1|^2}{(1+\epsilon|\phi_1|^2)^2}\end{pmatrix}
\]
Since $\phi_1$ vanishes at points of the $\Delta(p,q,r)$-orbit of $v_1$, we see that the first two entries of $R_{f\lambda}$ vanish at any point $(v,t)$ for $v$ in that orbit, and in fact $R_{f\lambda}(v,t) = R_{\lambda}(v,t) = \partial_t$ for such $v$. The induced contact form on $\Sigma(p,q,r)$, which we will abusively write as $f\lambda_{p,q,r}$, therefore has a closed Reeb orbit whose image is the $S^1$ fiber over the point $[v_1]\in S^2(p,q,r)$. Moreover, since a regular fiber in $\Sigma(p,q,r)$ corresponds to a generic vertical segment in $\hh\times \arr$ of length $2\pi d$, we have that this closed orbit has minimal period $2\pi d/p = 2\pi qr/s$ (this is essentially the statement that the element $g_1^d\in\Pi(p,q,r)$ has $(g_1^d)^p = c^d$; see the comments after \eqref{pises}).

Let us write $\gamma_p$ for the Reeb orbit of period ${2\pi d/p}$ lying over $[v_1]$. We wish to see that for generic small $\epsilon$, this orbit is nondegenerate (and similar for its iterates). The basic observation here is that the projection of the Reeb vector field to $\mathbb{H}$ is simply the symplectic gradient $X_{1/f}$ of $\frac{1}{f} = \frac{1}{1+\epsilon_1|\phi_1|^2}$. Indeed, since $f$ is independent of $t$, the projection of $R_{f\lambda}$ is well-defined and given by
\begin{equation}\label{Xdef}
X_{1/f} = \pi_*(R_{f\lambda}) = \begin{pmatrix} -y^2\partial_y(f^{-1}) \\ y^2\partial_x(f^{-1})\end{pmatrix},
\end{equation}
which is easily seen to satisfy 
\[
\iota_{X_{1/f}}\omega = -d(\ts f^{-1})
\]
for the symplectic form $\omega = \frac{1}{y^2} dx\wedge dy$ on $\hh$. Thus the Reeb flow of $f\lambda$ is a lift of the Hamiltonian flow associated to $f^{-1}$. Similarly, identifying $\xi_{(z,t)}$ with $T_z\hh$ using the symplectic isomorphism given by $d\pi$, the derivative of the time-$t$ flow of $R_{f\lambda}$ (strictly, the restriction of the derivative to the contact planes) is identified with the derivative of the time-$t$ flow of $X_{1/f}$.

\begin{lemma} For any integer $N>0$ there is finite set $\S_N\subset(0,1)$ such that for any $\epsilon\in (0,1) - \S_N$, and any $n\in\{1,\ldots, N\}$, the orbit $\gamma_1^n$ in $\Sigma(p,q,r)$ is a nondegenerate elliptic orbit of $R_{f\lambda}$, where $f = \epsilon |\phi_1|^2 +1$ as above.
\end{lemma}

\begin{proof} Let $T = 2\pi\frac{qr}{s} n$, so that $\gamma_1^n$ is a closed orbit of $R_{f\lambda}$ having period $T$ (independent of $\epsilon$). We wish to choose $\epsilon$ such that that the derivative of the flow map $\xi_{(v_1,0)}\to \xi_{(v_1,T)}$ does not have $1$ as an eigenvalue; here we use a vertical translation (lying in the group $\Pi(p,q,r)$) to identify $(v_1,0)$, $(v_1,T)$, and the corresponding point of $\Sigma(p,q,r)$. By the preceding remarks, this is equivalent to the same property holding for the derivative of the time-$T$ flow of $X_{1/f}$ at the point $v_1$ (fixed under the flow of $X_{1/f}$).

Let us consider $X_{1/f}$ as a smooth map $\hh\to \arr^2$. In other words, we work on the universal cover of the orbifold $S^2(p,q,r)$; this is justified since all our constructions are induced from equivariant ones on $\hh\times\arr$, and the analysis is local. If $\Phi_T: \hh\to\hh$ denotes the time $T$ flow of $X_{1/f}$, it is an exercise to see that $(d\Phi_T)_{v_1} = e^{TA}$, where $A = (dX_{1/f})_{v_1}$ is the linearization of $X_{1/f}$ at $v_1$. To compute this linearization, it is convenient to switch to the basis $\partial_z, \partial_{\bar z}$ for $T_\cee\hh$, using $\partial_x = \partial_z+\partial_{\overline{z}}$ and $\partial_y = i(\partial_z-\partial_{\overline{z}})$. We find that 
\[
X_{1/f} = y^2(2i\partial_{\bar z}(f^{-1}) \partial_z - 2i \partial_z(f^{-1})\partial_{\bar z}).
\]
Observe that with $f = 1 + |\phi_1|^2$, since $\phi_1$ has simple zero at $v_1$, the first derivatives of $f^{-1}$ vanish at $v_1$. Neglecting terms containing only the first derivatives of $f^{-1}$, then, the linearization of $X_{1/f}$ at $v_1$ becomes 
\[
\begin{pmatrix} d(2iy^2 \partial_{\bar z}(f^{-1})) \\ d(-2iy^2 \partial_z (f^{-1})) \end{pmatrix}_{v_1}
 = \begin{pmatrix}2iy^2 \partial_z\partial_{\bar z} (f^{-1}) & 2iy^2 \partial_{\bar z}\partial_{\bar z} (f^{-1}) \\ 
 - 2iy^2 \partial_z\partial_z (f^{-1}) & 
  -2iy^2 \partial_{\bar z}\partial_z (f^{-1})
 \end{pmatrix}_{v_1}
 \]
 Now compute at $v_1$ using that $\phi_1$ is holomorphic:
 \[
 \partial_z \partial_z(f^{-1})|_{v_1} = \partial_z(-(1+\epsilon|\phi_1|^2)^{-2}\epsilon\bar\phi_1\partial_z\phi) |_{v_1}= \begin{pmatrix}\mbox{terms all} \\\mbox{involving $\bar\phi_1(v_1)$} \end{pmatrix}= 0
\]
and similarly $\partial_{\bar z}\partial_{\bar z} (f^{-1})$ vanishes at $v_1$. On the other hand, at $v_1$ we have
\[
\partial_z\partial_{\bar z}(f^{-1}) = \partial_z(-(1+\epsilon|\phi_1|^2)^{-2} \epsilon\phi_1\partial_{\bar z}\bar\phi_1) = -\epsilon|\partial_z\phi_1|^2,
\]
and similarly for $\partial_{\bar z}\partial_{z}(f^{-1})$. Hence in the basis $\{\partial_z, \partial_{\bar z}\}$, 
\[
(dX_{1/f})_{v_1} = \begin{pmatrix} -2i(v_1^y)^2 \epsilon|\partial_z\phi_1|^2 & 0 \\
0 & 2i(v_1^y)^2\epsilon|\partial_z\phi_1|^2 \end{pmatrix}_{z = v_1},
\]
where we have written $v_1^y$ for the $y$-coordinate of the vertex $v_1$. Returning to the real basis $\{\partial_x, \partial_y\}$, this becomes 
\[
(dX_{1/f})_{v_1} = \begin{pmatrix} 0 & 2\epsilon (v_1^y)^2 |\partial_z\phi_1(v_1)|^2 \\
-2\epsilon (v_1^y)^2 |\partial_z\phi_1(v_1)|^2 & 0 \end{pmatrix}.
\]

The exponential of ($T$ times) this matrix, which gives the linearization of the flow of $X_{1/f}$, is easily seen to be a clockwise rotation proportional to $\epsilon$; specifically, $(d\Phi_T)_{v_1}$ is a clockwise rotation through the angle $2(v_1^y)^2|\partial_z\phi_1|^2T\epsilon$. Taking $T = 2\pi\frac{qr}{s}n$, there are clearly finitely many values of $\epsilon\in (0,1)$ that result in a trivial rotation for some value of $n\leq N$; all others give non-identity rotations and hence the corresponding $\gamma_1^n$ are nondegenerate.
\end{proof}

\subsection{Conley-Zehnder index of an exceptional orbit}

Continuing with the notation above, we wish to calculate the Conley-Zehnder index of the nondegenerate Reeb orbit $\gamma_1^N$ associated to the perturbed contact form $f\lambda$. From the preceding subsection we know that $\gamma_1^N$ is an elliptic orbit (the linearized return map is a rotation), so its Conley-Zehnder index will be odd. Strictly the value of $CZ(\gamma_1^N)$ depends on the choice of trivialization of $\xi$ along $\gamma_1^N$; here and throughout we will use the global trivialization coming from the basis $\{e_1,e_2\}$ described in Proposition \ref{trivprop}: since the basis is $\SLt$-invariant, it descends to a trivialization of $\xi$ on $\Sigma(p,q,r)$. In fact, it is equivalent to work in the universal cover $\hh\times \arr$ and consider the linearized flow along the path $(v_1,t)$, $t\in [0,T]$, where $T = 2\pi\frac{qr}{s}N$.

With the work we've already done, the calculation is fairly trivial. We can identify the contact planes in $\hh\times \arr$ with the tangent planes to $\hh\times\{0\}$ via a vertical projection;  we've seen that the  time-$t$ Reeb flow projects to the time-$t$ Hamiltonian flow of $X_{1/f}$, which, in the ``obvious'' (non-invariant but constant) basis $\{\partial_x, \partial_y\}$, has linearization that is a clockwise rotation proportional to $\epsilon t$. In particular, choosing $\epsilon\ll (2(v_1^y)^2|\partial_z\phi_1|^2T)^{-1}$, the Reeb flow in this trivialization completes no full revolutions at all. On the other hand, lifting to $\hh\times \arr$ the invariant basis $\{ e_1,e_2\}$ rotates (compared to the lift of $\{\partial_x,\partial_y\}$) through the angle $T$ as $t$ ranges from $0$ to $T$. Therefore with respect to $\{e_1,e_2\}$ the Hamiltonian flow revolves through an angle of $-T-\epsilon'(T)$
where $\epsilon'(T)>0$ depends on $T$ but is bounded above in terms of $\epsilon$.

\begin{proposition}\label{excepCZ} Let $N>0$ be an integer. Then there exists $\epsilon_0>0$ such that for all $0<\epsilon< \epsilon_0$ the curve $\gamma_1^n$ is a nondegenerate Reeb orbit for the contact form $f\lambda$ on $\Sigma(p,q,r)$ for all $n=1,\ldots, N$, where $f = 1+\epsilon |\phi_1|^2$ as above. Moreover, the Conley-Zehnder index with respect to the universal trivialization $\tau$ of Proposition \ref{trivprop} is given by
\[
CZ_\tau(\gamma_1^n) = -2\lfloor n\ts\frac{qr}{s}\rfloor - 1.
\]

\end{proposition}

Indeed, let $T = 2\pi\frac{qr}{s}N$ and choose $\epsilon$ correspondingly small as above. If $t = 2\pi\frac{qr}{s}n$, then the definition reads $CZ_{\tau}(\gamma_1^n) = 2\lfloor\theta_t\rfloor + 1$, and in our case $\theta_t = -n\frac{qr}{s} -\frac{\epsilon'}{  2\pi}$ (with $\epsilon'$ depending on $t$ but bounded above in terms of $\epsilon$ as before).  We can select $\epsilon$ small enough that for $n\leq N$  we have $\lfloor-n\frac{qr}{s} -\frac{\epsilon'}{2\pi}\rfloor = \lceil -n\frac{qr}{s} - 1\rceil = -\lfloor n\frac{qr}{s} +1\rfloor = -\lfloor n\frac{qr}{s}\rfloor - 1$.

Clearly, constructions analogous to those above using the functions $\phi_2$ and $\phi_3$ give rise to perturbations of the contact form on $\Sigma(p,q,r)$ such that the exceptional Seifert fibers over $[v_2]$ and $[v_3]$, respectively, (and their iterates up to some fixed multiple) become nondegenerate Reeb orbits of the perturbed form, and we have corresponding formulas for the Conley-Zehnder indices. We now wish to perturb $\lambda_{p,q,r}$ in such a way as to preserve this structure near the exceptional fibers.

\subsection{Global perturbation and Conley-Zehnder indices}\label{globalsec}

It is a simple matter to construct a function $H: S^2(p,q,r)\to [0,1]$ with the following properties:
\begin{enumerate}
\item In a small neighborhood of the orbifold point $[v_j]$, $H$ agrees with $|\phi_j|^2$.
\item In the complement of the orbifold points $H$ is a Morse function having a single maximum and two saddle critical points, and no other critical points.
\end{enumerate}
Fixing such $H$, we will consider the perturbed contact form $f\lambda_{p,q,r}$, where $f = 1 + \epsilon \pi^*H$ and $\pi: \Sigma(p,q,r)\to S^2(p,q,r)$ is the projection. This construction has been well-studied in the case of prequantization bundles (cf. Nelson \cite{Ne18}, for example), and in fact we can bootstrap that study to our situation. 

Recall that we have a finite (orbifold) covering map $\BtoS2: B(p,q,r)\to S^2(p,q,r)$ where $B(p,q,r)$ is a smooth surface, and a corresponding finite covering $\EtoSg: E(p,q,r)\to \Sigma(p,q,r)$. It follows from the construction at the beginning of this section that $\EtoSg$ restricts to a trivial covering over generic circle fibers of $\Sigma(p,q,r)$. It is also easy to see that ${\BtoS2}^*H$ is a Morse function on $B(p,q,r)$ (this is clear away from the preimages of orbifold points, which is all we need, but is also true on those preimages since on $\hh$ the $\phi_j$  have simple zeros at the $v_j$). Moreover the critical points of $\BtoS2^*H$ are the preimages of the critical points of $H$ together with the preimages of the orbifold points, and the indices are preserved (the preimages of orbifold points are local minima). Finally, observe that $E(p,q,r)$ is a prequantization bundle: that is, as a principal $S^1$ bundle it admits a connection 1-form $\lambda$ such that $d\lambda$ is the pullback of a symplectic form on the base. Indeed, our $\SLt$-invariant contact form on $\hh\times\arr$ descends to just such a form, where the corresponding symplectic form is induced by $\frac{1}{y^2}dx\wedge dy$.

Now, the Reeb orbits of the contact form on a prequantization bundle are exactly the circle fibers, and in particular all are degenerate. Perturbing using a Morse function on the base produces a nondegenerate contact form whose closed Reeb orbits (up to a given action) are the circle fibers over the critical points of the Morse function. For the purposes of the following result we will write $\gamma_z$ for the closed orbit of the unperturbed contact form lying over a critical point $z$, and $\gamma_{z,\epsilon}$ for the same orbit considered as a nondegenerate orbit of the perturbed contact form.  

\begin{theorem}\label{IndexFormula} (\cite[Theorem 4.1]{Ne18}) Let $(V,\lambda)$ be a prequantization bundle over a closed symplectic surface $(\Sigma,\omega)$.  Fix a Morse function $H$ on $\Sigma$ such that $|H|_{C^2}<1$, and let $T>0$.  There exists $\epsilon>0$ such that all Reeb orbits $\gamma$ of $(1+\epsilon\pi^*H)\lambda$ with $\mathcal{A}(\gamma)<T$ are non-degenerate and project to critical points of $H$.  For a critical point $z$, supposing $\mathcal{A}(\gamma^N_z)<T$,
\begin{equation}\label{CZRSeqn}
 CZ(\gamma_{z,\epsilon}^N) = RS(\gamma_z^N)-1+\operatorname{index}_z(H)
 \end{equation}
where $\gamma_z^N$ is the $N$th iterate of the simple  Reeb orbit of $\lambda$ corresponding to the fiber over $z$, and $RS(\gamma^N_z)$ is its Robbin-Salamon index.
\end{theorem}

The utility of this result is that the Robbin-Salamon index of an orbit $\gamma$ is defined even when $\gamma$ is degenerate, and is easily computed in the case of a prequantization bundle. Note that the two sides of \eqref{CZRSeqn} must be computed with respect to the same trivialization of the contact structure; we use the universal trivialization described above instead of the ``constant'' trivialization used in \cite{Ne18}. 

We do not review the definition of $RS(\gamma)$ here: we need only the fact proved in \cite[Example 4.4]{Ne18} that when the linearized flow along $\gamma$ is given by rotation by $t$ for $t\in[0,2\pi n]$, then $RS(\gamma) = 2n$. In our case, the simple orbit over $z$ (a critical point of $\BtoS2^*H$ other than a preimage of an orbifold point) is modeled on a generic vertical segment in $\hh\times\arr$ of length $2\pi d$, and the unperturbed flow is simply vertical translation. As we have seen, however, the invariant trivialization rotates counterclockwise through an angle of $2\pi d$ along this segment, so with respect to this trivialization the linearized flow appears as a {\it negative} rotation by $t$, for $t\in [0,2\pi d]$. Thus $RS(\gamma_z^N) = -2 d N = -2N\frac{pqr}{s}$.

\begin{corollary}\label{gradingcor}  Let $H: S^2(p,q,r)\to [0,1]$ and $f = 1 + \epsilon \pi^* H$ be as at the beginning of this subsection, and fix $T>0$. Then for all sufficiently small $\epsilon$, any Reeb orbit of the contact form $f\lambda_{p,q,r}$ having action less than $T$ is an iterate of one of the six simple orbits $\gamma_{v_1}, \gamma_{v_2}, \gamma_{v_3}, \gamma_{s_1}$, $\gamma_{s_2}$, $\gamma_m$, where
\begin{enumerate}
\item The orbits $\gamma_{v_1}, \gamma_{v_2}, \gamma_{v_3}$ are the exceptional Seifert fibers of order $p$, $q$, and $r$ respectively.
\item The orbits $ \gamma_{s_1}$, $\gamma_{s_2}$, $\gamma_m$ are the regular Seifert fibers lying over the index 1 critical points $s_1$ and $s_2$ and over the maximum $m$ of $H$. 
\end{enumerate}

Moreover, the Conley-Zehnder indices of the iterates of these orbits are given by 
\[
\ts CZ(\gamma_{v_1}^n) = -2\lfloor n \frac{qr}{s}\rfloor - 1, \quad CZ(\gamma_{v_2}^n) = -2\lfloor n \frac{pr}{s}\rfloor - 1, \quad CZ(\gamma_{v_3}^n) = -2\lfloor n \frac{pq}{s}\rfloor - 1
\]
and 
\[
\ts CZ(\gamma_{s_j}^n) = -2n\frac{pqr}{s}, \quad CZ(\gamma_m^n) = -2n\frac{pqr}{s} + 1.
\]
\end{corollary}

Note that there are no bad Reeb orbits: the only hyperbolic orbits are those corresponding to the saddles $s_j$, which have even Conley-Zehnder index and are therefore positive hyperbolic.

Recall that the grading of a generator $\gamma$ of the cylindrical contact chain complex is given by $CZ(\gamma ) - 1$. It follows that the grading of all iterates of the exceptional Seifert orbits (thought of as Reeb orbits of $\lambda_\epsilon$, where we restrict to action less than some $T$ and $\epsilon$ is chosen correspondingly small) is {\it even}, so there can be no differentials between these generators. 

In fact, we can give a complete description of the subcomplex $CC^T(\Sigma(p,q,r),\lambda_\epsilon, J_\epsilon)$ generated by Reeb orbits of action no more than $T$, using the results of \cite{Ne18}. For the following we assume that $H: S^2(p,q,r)\to [0,1]$ has been chosen so that the two downward gradient trajectories starting from the index-1 critical point $s_1$ have limits on $[v_1]$ and $[v_2]$, while the downward trajectories from $s_2$ limit on $[v_2]$ and $[v_3]$.

\begin{theorem}\label{boundarythm} Fix $H$ and $T$ as above. Then with suitable orientation conventions, for all sufficiently small $\epsilon$ the only nontrivial differentials in the complex $CC^T(\Sigma(p,q,r), \lambda_\epsilon, J_\epsilon)$ are given by
\[
\partial (\gamma_{s_1}^n) = \gamma_{v_1}^{np} - \gamma_{v_2}^{nq} \quad\mbox{and}\quad \partial(\gamma_{s_2}^n) = \gamma_{v_2}^{nq} - \gamma_{v_3}^{nr}
\]
for each $n$ with $\mathcal{A}(\gamma_{s_i}^n) < T$.
\end{theorem}

\begin{proof} The differential in $CC^T$ counts $J$-holomorphic cylinders in $\arr\times \Sigma(p,q,r)$, which are asymptotic to Reeb orbits $\gamma_\pm$ in the same free homotopy class. From the previous corollary, we can assume that this is the homotopy class of (a multiple of) a regular fiber in the Seifert fibration of $\Sigma(p,q,r)$ (note that among the iterates of the singular fibers, exactly $\gamma_{v_1}^{np}$, $\gamma_{v_2}^{nq}$ and $\gamma_{v_3}^{nr}$ lie in such a free homotopy class, for $n\in \zee$). 

As described above, there is a finite covering $E(p,q,r)\to \Sigma(p,q,r)$, under which our setup lifts to the setup of Nelson \cite{Ne18} on a prequantization bundle, albeit with a finite symmetry. In particular, a holomorphic cylinder $u: \arr\times S^1\to \arr\times \Sigma(p,q,r)$ lifts to a holomorphic map $\tilde{u} : S\to \arr\times E(p,q,r)$, where $S$ is a suitable cover of $\arr\times S^1$. Since $u$ is in the homotopy class of a multiple of a regular fiber, and the cover is trivial on regular fibers, the domain of the lifted curve $\tilde{u}$ is a disjoint union of cylinders. It suffices to consider an individual holomorphic cylinder $\tilde{u} : \arr\times S^1\to \arr\times E(p,q,r)$, satisfying $(id\times\EtoSg)\circ \tilde{u} = u$.

According to \cite[Theorems 5.1 and 5.5]{Ne18}, for given $T$ and sufficiently small $\epsilon$, every such curve $\tilde{u}$ is equivalent to a cylinder lying over a downward gradient trajectory of $\epsilon \widetilde{H}$ (where $\widetilde{H} = \BtoS2^*H$); this result does not require the almost complex structure on $\arr\times E(p,q,r)$ to be generic. In fact, there is a 1-1 correspondence between gradient trajectories and equivalence classes of holomorphic cylinders, at least in the relevant case of index difference 1 Reeb orbits or critical points. Hence given $u$, the lifted union of cylinders $\tilde{u}$ lies over a union of gradient trajectories which, by equivariance, descend to a unique gradient trajectory of $\epsilon H$ on $S^2(p,q,r)$. By choosing signs compatibly with those in the Morse complex as in \cite[Section 5.2]{Ne18}, and noting that in the Morse complex the boundary of the maximum point $m$ is zero, we obtain the result.

\end{proof}

The previous result describes the cylindrical contact chain complex up to a given action $T$. However, the action and index of generators are proportional in our situation, so a spectral sequence argument as in \cite{Ne18} shows that the full cylindrical contact homology can be computed as the limit of the homologies of the complexes $CC^T$. The latter is easy to describe using the results above. By way of notation, define the vector space
\[
G(p,q,r) = \bigoplus_{k=1}^{p-1} \cue\langle \gamma_{v_1}^k\rangle \oplus \bigoplus_{k=1}^{q-1} \cue\langle \gamma_{v_2}^k\rangle \oplus \bigoplus_{k=1}^{r-1} \cue\langle \gamma_{v_3}^k\rangle
\]
spanned by the indicated iterates of the exceptional Seifert fibers. By Theorem \ref{boundarythm} this subspace of $CC(\Sigma(p,q,r))$ does not interact with the boundary operator, hence can be considered as lying in the homology. Clearly it is has dimension $(p-1) + (q-1) + (r-1)$, and according to Corollary \ref{gradingcor} it is graded with nonzero summands in only even gradings between $-2$ and $-2d$ (recall that $d = \frac{pqr}{s} = pqr - qr-pr-pq$). Moreover, observe that for $n\geq 1$, replacing $\gamma_{v_1}^k$ by $\gamma_{v_1}^{k+np}$ (and similar for $\gamma_{v_2}$ and $\gamma_{v_3}$) in the definition of $G(p,q,r)$, we obtain an isomorphic vector space with grading shifted by $-2nd$ lying in the contact homology. 

The remaining part of the homology is that corresponding to the free homotopy class of an iterate of a regular fiber. By Theorem \ref{boundarythm}, if we concentrate on the class of a single fiber, the resulting complex has three generators in dimension $-2d-2$, two in dimension $-2d-1$, and one in dimension $-2d$. The boundary of the latter is trivial while the boundary map is injective on the summand in dimension $-2d-1$. Thus the homology of this subcomplex is isomorphic to the singular homology $H_*(S^2)$ with grading shifted by $-2d-2$. The subcomplexes corresponding to multiples of the regular fiber are isomorphic, with additional grading shift. This proves the following result, a restatement of Theorem \ref{firstthm} of the introduction.

\begin{theorem} Let $p,q,r$ be relatively prime positive integers satisfying $\frac{1}{p} + \frac{1}{q} + \frac{1}{r} < 1$. The cylindrical contact homology of $\Sigma(p,q,r)$ with its standard universally tight contact structure is given by:
\[
HC_*(\Sigma(p,q,r), \xi_{p,q,r}) = \bigoplus_{n\geq 1} G(p,q,r)[-2nd] \oplus \bigoplus_{n\geq 1} H_*(S^2)[-2nd-2].
\]
\end{theorem}

\section{Brieskorn complete intersections}\label{gensec}

The calculations of the preceding sections carry over with little modification to the case of a general link $\Sigma(a_1,\ldots, a_n)$ of a Brieskorn complete intersection singularity \eqref{brieskorndef} in light of the results of Neumann \cite{neumann}. Here we allow $(a_1,\ldots, a_n)$ to be any $n$-tuple of integers, $a_j \geq 2$, subject to the condition that $\sum \frac{1}{a_j} < n-2$. In this case the results of section \ref{Bri2} generalize as follows. 

Let $P$ be a hyperbolic polygon with interior angles $\pi/a_1,\ldots, \pi/a_n$, and $\Delta = \Delta(a_1,\ldots,a_n)$ the orientation-preserving subgroup of the group of isometries generated by reflections in the edges of $P$. We consider the central extension $\tilde\Delta = \tilde\Delta(a_1,\ldots,a_n) \subset \SLt$ which is the preimage of $\Delta(a_1,\ldots, a_n)$ under the universal covering homomorphism $\SLt\to PSL(2,\arr)$, and let
\[
\Pi = \Pi(a_1,\ldots,a_n) = [\tilde\Delta,\tilde\Delta]\subset \SLt.
\]
The sequence \eqref{pises} becomes
\[
1\to C_0 \to \Pi \to [\Delta,\Delta]\to 1,
\]
where $C_0$ is the infinite cyclic center of $\Pi$ and the commutator subgroup $[\Delta, \Delta]\subset\Delta$ has index 
\[
m = \gcd_j(\Pi_{i\neq j} a_i). 
\]
Indeed, the index is not hard to compute from the presentation
\[
\Delta = \langle \delta_1,\ldots, \delta_n \, | \, \delta_1^{a_1} = \cdots = \delta_n^{a_n} =\delta_1\cdots\delta_n = 1\rangle.
\]
In particular, the induced action of $\Pi$ on the hyperbolic plane has a fundamental domain that is the union of $2m$ copies of the standard polygon $P$.

Similarly, the index of $\Pi$ in $\tilde\Delta$ is given by the absolute value of the determinant of the relation matrix in its presentation,
\[
\tilde\Delta = \langle \gamma_1,\ldots, \gamma_n, c\,|\, \gamma_1^{a_1} = \cdots = \gamma_n^{a_n} = c,\,\gamma_1\cdots\gamma_n = c^{n-2}\rangle.
\]
The latter is easily calculated and gives
\begin{equation}\label{ddef}
[\tilde\Delta\,:\, \Pi] = d =\ts (\prod_j a_j)(n-2-\sum_j \frac{1}{a_j}),
\end{equation}
which reduces to our earlier definition for $d$ if $n = 3$.

Now, a fundamental domain $D_{\tilde\Delta}$ for the action of $\tilde\Delta$ on $\SLt \cong\hh\times \arr$ can be identified with $(P\cup P')\times [0,2\pi]$, where $P'$ is a reflection of $P$ across an edge, and a fundamental domain $D_\Pi$ for $\Pi$ contains $d$ of these polygonal prisms. (Note that translation of $(P\cup P')\times [0,2\pi]$ by an element $\tilde\delta\in\tilde\Delta$, lifting a nontrivial element of $\Delta$, preserves the vertical segments but also shifts them up or down by varying amounts according to the rotation induced by the derivative of the action of $\tilde\delta$ on $\hh$.) As $D_\Pi$ projects to a fundamental domain $D_{[\Delta,\Delta]}$ for $[\Delta,\Delta]$, containing $m$ copies of $P\cup P'$, we see that over each copy of $P\cup P'$, the domain $D_\Pi$ contains $d/m$ copies of $D_{\tilde\Delta}$ ``stacked'' vertically. In particular we have an identification $\Pi\setminus \hh\times\arr \cong D_\Pi/\sim$, where $\sim$ indicates face identifications determined by $\Pi$, and vertical translations induce an action of $S^1$ on $\Pi\setminus\hh\times\arr$ whose regular fibers have length $2\pi d/m$.

Let $S = D_{[\Delta,\Delta]}$ be the result of making the edge identifications on the polygon $D_{[\Delta,\Delta]}$ induced by the action of $[\Delta,\Delta]$. Then $S$ is a closed surface of genus $g$ given by \eqref{genusS} (see \cite{neumann,neumannraymond}), smooth but with distinguished orbifold points arising from the vertices of $D_{[\Delta,\Delta]}$ (analogous to the orbifold $S^2(p,q,r)$ of Section \ref{Bri2}). The vertical segments over vertices $v$ of $D_{[\Delta,\Delta]}$ give rise to the exceptional fibers in the Seifert fibration
\[
\Pi\setminus \hh\times \arr \to S,
\]
but as differing numbers of vertices can be identified to different orbifold points, the number of orbifold points and their corresponding multiplicities take some working out. One finds that for $j = 1,\ldots, n$ there are $s_j$ orbifold points of multiplicity $t_j$, where
\[
s_j = \frac{\prod_{i\neq j} a_i}{\lcm_{i\neq j} a_i} \quad\mbox{and}\quad t_j = \frac{\lcm (a_i)}{\lcm_{i\neq j} a_i}.
\]

The Seifert fiber over each such orbifold point corresponds to a vertical segment over a vertex of $D_{[\Delta,\Delta]}$ such that the action of $\Pi$ identifies points on the segment that are separated by less than $2\pi d/m$. Specifically, the Seifert fiber over a point of multiplicity $t_j$ corresponds to a vertical segment of length $2\pi d/m t_j$.

Milnor's argument that $\Pi(p,q,r)\backslash \hh\times\arr \cong \Sigma(p,q,r)$ carries over with only minor modifications to show that $\Pi(a_1,\ldots,a_n) \backslash \hh\times\arr \cong \Sigma(a_1,\ldots, a_n)$, and as before the identification is induced by the construction of $\Pi(a_1,\ldots,a_n)$-automorphic forms on $\hh\times \cee$ (for details, see Neumann \cite{neumann77}). In particular, the contact form on $\Sigma(a_1,\ldots, a_n)$ induced by the invariant form $\lambda$ on $\hh\times \arr$ from Theorem \ref{invcontact} gives a contact structure isotopic to the one induced by thinking of $\Sigma(a_1,\ldots, a_n)$ as the link of the surface singularity as in \eqref{brieskorndef} (the Milnor fillable contact structure), by an argument analogous to that of section \ref{contstrident}. It follows just as in that section that the Milnor fillable contact structure is hypertight, when $\sum \frac{1}{a_j} < n-2$. 

Now we introduce a perturbation of the contact form to obtain nondegeneracy. As in section \ref{globalsec}, we can find a function $H: S\to [0,1]$ on the base of the Seifert fibration of $\Sigma(a_1,\ldots, a_n)$ that is  Morse away from small standard neighborhoods of the orbifold points of $S$. Each orbifold point corresponds to (an orbit of) a zero of the smooth map $|\phi_j|^2$ on $\hh$, where $\phi_j$ is an automorphic form on $\hh\times\wt{\cee^\times}$ as before, and since the restriction of $|\phi_j|^2$ to $\hh\times \arr$ is independent of the $\arr$ coordinate, it induces a function on $S$ (strictly, we are identifying $\hh\times \arr$ with its image in $\hh\times\wt{\cee^\times}$, see \eqref{liftedunit} and surrounding text). We can arrange that the function $H$ agrees with $|\phi_j|^2$ in a neighborhood of the corresponding orbifold point, and that the orbifold points are exactly the set of minima of $H$. 

Let $\{v_{ij}\}$, $i = 1,\ldots, s_j$, $j = 1,\ldots, n$   be the orbifold points on $S$ and let $x_1,\ldots, x_r, y$ be the ordinary critical points of $H$, where we may assume that $y$ is the unique critical point of index 2 and $x_1,\ldots, x_r$ have index 1. The analog of Corollary \ref{gradingcor} in this situation is:

\begin{corollary} Let $H: S^2(p,q,r)\to [0,1]$ and $f = 1 + \epsilon \pi^* H$, where $\pi: \Sigma(a_1,\ldots, a_n)\to S$ is the Seifert projection. Then for all sufficiently small $\epsilon$, any Reeb orbit of the contact form $f\lambda$ having action less than $T$ is an iterate of one of the  simple orbits $\gamma_{v_{ij}}, \gamma_{x_\ell}$, $\gamma_y$, where
\begin{enumerate}
\item The orbit $\gamma_{v_{ij}}$ is the exceptional Seifert fiber lying over $v_{ij}$.
\item The orbits $ \gamma_{x_\ell}$, $\ell = 1,\ldots, r$ and $\gamma_y$ are the regular Seifert fibers lying over the index 1 critical points $x_\ell$ and over the maximum $y$ of $H$. 
\end{enumerate}

Moreover, the Conley-Zehnder indices of the iterates of these orbits are given by 
\[
\ts CZ(\gamma_{v_{ij}}^n) = -2\lfloor \frac{nd}{mt_j}\rfloor - 1, 
\]
and 
\[
\ts CZ(\gamma_{x_\ell}^n) = -\frac{2nd}{m}, \quad CZ(\gamma_y^n) = -\frac{2nd}{m} + 1.
\]
\end{corollary}

Indeed, the form $f\lambda$ has (up to fixed action and for small enough $\epsilon$) only nondegenerate Reeb orbits over the orbifold points and critical points of $H$. Over an orbifold point, the Conley-Zehnder index is given by the analog of Proposition \ref{excepCZ}, where we replace $\frac{pqr}{s}$ by the number $\frac{d}{m}$ of full rotations in a regular orbit; similar considerations using the Robbin-Salamon index as in Section \ref{globalsec} give the Conley-Zehnder indices of orbits over ordinary critical points.

Just as in Section \ref{globalsec}, degree considerations show that only the subcomplexes of the cylindrical contact chain complex corresponding to iterates of a regular fiber can have nontrivial differentials. Moreover, that complex is identified with the Morse complex on (the smooth surface underlying) $S$, by an argument entirely analogous to the one proving Theorem \ref{boundarythm}. We obtain the following general result, which reduces to Theorem \ref{introgenthm} if the $a_j$ are pairwise relatively prime.

\begin{theorem}\label{generalthm} Let $a_1,\ldots, a_n$ be integers, $a_j\geq 2$, satisfying $\sum \frac{1}{a_j} < n-2$. The cylindrical contact homology of the Brieskorn manifold $\Sigma(a_1,\ldots,a_n)$ with its standard contact structure $\xi_{std}$ is given by:
\[
HC_*(\Sigma(a_1,\ldots, a_n), \xi_{std}) = \bigoplus_{n\geq 1} {\ts G(a_1,\ldots, a_n)[-\frac{2nd}{m}]} \oplus \bigoplus_{n\geq 1} {\ts H_*(S)[-\frac{2nd}{m}-2]}
\]
where $G(a_1,\ldots, a_n)$ is the part generated by the exceptional orbits, given by
\[
G(a_1,\ldots, a_n) = \bigoplus_{{i = 1,\ldots, s_j}\atop {j = 1,\ldots n}}\bigoplus_{k=1}^{t_j-1} \cue\langle \gamma_{v_{ij}}^k\rangle,
\]
graded such that $\gamma_{v_{ij}^k}$ has grading $-2\lfloor \frac{kd}{mt_j}\rfloor - 2$, and the (ordinary) homology group $H_*(S)$ carries its natural grading. 

The grading on $HC_*$ should be understood as calculated with respect to the universal trivialization of the contact structure described above.
\hfill$\Box$
\end{theorem}

\bibliographystyle{plain}
\bibliography{mybib}

\begin{thebibliography}{10}

\bibitem{dolgacev}
I.~V. Dolga\v{c}ev.
\newblock Conic quotient singularities of complex surfaces.
\newblock {\em Funkcional. Anal. i Prilo\v{z}en.}, 8(2):75--76, 1974.

\bibitem{EGH}
Y.~Eliashberg, A.~Givental, and H.~Hofer.
\newblock Introduction to symplectic field theory.
\newblock {\em Geom. Funct. Anal.}, (Special Volume, Part II):560--673, 2000.
\newblock GAFA 2000 (Tel Aviv, 1999).

\bibitem{fox52}
Ralph~H. Fox.
\newblock On {F}enchel's conjecture about {$F$}-groups.
\newblock {\em Mat. Tidsskr. B}, 1952:61--65, 1952.

\bibitem{geigesbook}
Hansj\"{o}rg Geiges.
\newblock {\em An introduction to contact topology}, volume 109 of {\em
  Cambridge Studies in Advanced Mathematics}.
\newblock Cambridge University Press, Cambridge, 2008.

\bibitem{HutchingsECHnotes}
Michael Hutchings.
\newblock {Lecture notes on embedded contact homology}.
\newblock https://arxiv.org/abs/1303.5789.

\bibitem{HN2019}
Michael Hutchings and Jo~Nelson.
\newblock {$S^1$-equivariant contact homology for hypertight contact forms}.
\newblock http://arxiv.org/abs/1906.03457.

\bibitem{HN2016}
Michael Hutchings and Jo~Nelson.
\newblock Cylindrical contact homology for dynamically convex contact forms in
  three dimensions.
\newblock {\em J. Symplectic Geom.}, 14(4):983--1012, 2016.

\bibitem{milnor1975}
John Milnor.
\newblock On the {$3$}-dimensional {B}rieskorn manifolds {$M(p,q,r)$}.
\newblock In {\em Knots, groups, and 3-manifolds ({P}apers dedicated to the
  memory of {R}. {H}. {F}ox)}, pages 175--225. Ann. of Math. Studies, No. 84.
  1975.

\bibitem{Ne18}
Jo~Nelson.
\newblock {Automatic transversality in contact homology II: filtrations and
  computations}.
\newblock http://arxiv.org/abs/1708.07220.

\bibitem{neumann}
W.~Neumann.
\newblock Graph 3-manifolds, splice diagrams, singularities.
\newblock In {\em Singularity theory}, pages {787--817}. {World Sci. Publ.,
  Hackensack, NJ}, {2007}.

\bibitem{neumann77}
Walter~D. Neumann.
\newblock Brieskorn complete intersections and automorphic forms.
\newblock {\em Invent. Math.}, 42:285--293, 1977.

\bibitem{neumannraymond}
Walter~D. Neumann and Frank Raymond.
\newblock Seifert manifolds, plumbing, {$\mu $}-invariant and orientation
  reversing maps.
\newblock In {\em Algebraic and geometric topology ({P}roc. {S}ympos., {U}niv.
  {C}alifornia, {S}anta {B}arbara, {C}alif., 1977)}, volume 664 of {\em Lecture
  Notes in Math.}, pages 163--196. Springer, Berlin, 1978.

\bibitem{scott83}
Peter Scott.
\newblock The geometries of {$3$}-manifolds.
\newblock {\em Bull. London Math. Soc.}, 15(5):401--487, 1983.

\bibitem{selberg}
Atle Selberg.
\newblock On discontinuous groups in higher-dimensional symmetric spaces.
\newblock In {\em Contributions to function theory (internat. {C}olloq.
  {F}unction {T}heory, {B}ombay, 1960)}, pages 147--164. Tata Institute of
  Fundamental Research, Bombay, 1960.

\end{thebibliography}

\end{document}